\documentclass[11pt]{amsart}
\usepackage{amssymb,amsmath}
\usepackage{amsmath,amscd}
\usepackage{verbatim}
\usepackage{enumerate}
\usepackage[hidelinks,pdftex]{hyperref}
\usepackage{overpic}
 \let\oldmarginpar\marginpar
 \renewcommand\marginpar[1]{\oldmarginpar[\raggedleft\footnotesize #1]%
 {\raggedright\footnotesize #1}}

\numberwithin{equation}{section}
\theoremstyle{plain}
\newtheorem{theorem}{Theorem}
\numberwithin{theorem}{section}

\newtheorem{lemma}[theorem]{Lemma}
\newtheorem{observation}[theorem]{Observation}

\newtheorem{proposition}[theorem]{Proposition}

\newtheorem{claim}[theorem]{Claim}

\newtheorem*{namedtheorem}{\theoremname}
\newcommand{\theoremname}{testing}
\newenvironment{named}[1]{\renewcommand{\theoremname}{#1}\begin{namedtheorem}}{\end{namedtheorem}}

\theoremstyle{definition}
\newtheorem{definition}[theorem]{Definition}
\newtheorem{remark}[theorem]{Remark}
\newtheorem{question}[theorem]{Question}

\newcommand{\HH}{{\mathbb{H}}}
\newcommand{\RR}{{\mathbb{R}}}
\newcommand{\ZZ}{{\mathbb{Z}}}
\newcommand{\NN}{{\mathbb{N}}}
\newcommand{\CC}{{\mathbb{C}}}
\newcommand{\QQ}{{\mathbb{Q}}}

\newcommand{\calC}{{\mathcal{C}}}
\newcommand{\calL}{{\mathcal{L}}}
\newcommand{\calO}{{\mathcal{O}}}
\newcommand{\calQ}{{\mathcal{Q}}}

\newcommand{\Teich}{\mathcal{T}}
\newcommand{\Curve}{\mathcal{C}}
\newcommand{\Pants}{\mathcal{P}}
\newcommand{\core}{{\operatorname{core}}}

\newcommand{\refthm}[1]{Theorem~\ref{Thm:#1}}
\newcommand{\reflem}[1]{Lemma~\ref{Lem:#1}}
\newcommand{\refprop}[1]{Proposition~\ref{Prop:#1}}

\newcommand{\refclaim}[1]{Claim~\ref{Claim:#1}}

\newcommand{\refques}[1]{Question~\ref{Ques:#1}}
\newcommand{\refobs}[1]{Observation~\ref{Obs:#1}}
\newcommand{\refeqn}[1]{\eqref{Eqn:#1}}
\newcommand{\refitm}[1]{\eqref{Itm:#1}}
\newcommand{\refdef}[1]{Definition~\ref{Def:#1}}
\newcommand{\refsec}[1]{Section~\ref{Sec:#1}}
\newcommand{\reffig}[1]{Figure~\ref{Fig:#1}}
\newcommand{\refrem}[1]{Remark~\ref{Rem:#1}}

\newcommand{\bdy}{\partial}
\newcommand{\isom}{\cong}
\renewcommand{\setminus}{\smallsetminus}

\newcommand{\vol}{\operatorname{vol}}
\newcommand{\injrad}{{\operatorname{injrad}}}

\newcommand{\Isom}{\operatorname{Isom}}

\newcommand{\area}{\operatorname{area}}

\newcommand{\li}{\operatorname{li}}
\newcommand{\rad}{\operatorname{rad}}

\begin{document}

\title{Spectrally similar incommensurable $3$--manifolds}

\author{David Futer}
\address{Department of Mathematics, Temple University\\ Philadelphia, PA 19122}
\email{dfuter@temple.edu}

\author{Christian Millichap} 
\address{Department of Mathematics, Linfield College\\ McMinnville, OR 97128}
\email{Cmillich@linfield.edu}

\thanks{{Futer was supported in part by NSF grant DMS--1408682 and the Elinor Lunder Founders' Circle Membership at the Institute for Advanced Study.}}

\date{\today}

\begin{abstract}
Reid has asked whether hyperbolic manifolds with the same geodesic length spectrum must be commensurable. Building toward a negative answer to this question, we construct examples of hyperbolic $3$--manifolds that share an arbitrarily large portion of the length spectrum but are not commensurable. More precisely, for every $n \gg 0$, we construct a pair of incommensurable hyperbolic $3$--manifolds $N_n$ and $N_n^\mu$ whose volume is approximately $n$ and whose length spectra agree up to length $n$. 

Both $N_n$ and $N_n^\mu$ are built by gluing two standard submanifolds along a complicated pseudo-Anosov map, ensuring that these manifolds have a very thick collar about an essential surface. The two gluing maps differ by a hyper-elliptic involution along this surface. Our proof also involves a new commensurability criterion based on pairs of pants.
\end{abstract}

\maketitle
\section{Introduction}
\label{Sec:Intro}

% \subsection{Setup}

This paper is devoted to the following question, posed and studied by Reid \cite{Reid:GeomTopArithmetic, Reid:traces-lengths}:

\begin{question}\label{Ques:Reid}
Let $M_1 = \HH^n / \Gamma_1$ and $M_2 = \HH^n / \Gamma_2$ be hyperbolic manifolds of finite volume. If the length spectra of the $M_i$ agree, must $M_1$ and $M_2$ be commensurable?
\end{question}

All manifolds appearing in this paper are presumed to be orientable. 
The \emph{length spectrum} of a manifold $M$, denoted $\calL(M)$, is the ordered tuple of all lengths of closed geodesics in $M$, counted with multiplicity. Meanwhile, $M_1$ and $M_2$ are called \emph{commensurable} if some hyperbolic manifold $\hat{M}$ serves as a common finite-sheeted cover of both $M_1$ and $M_2$. Commensurability is an equivalence relation, and the equivalence class containing $M$ is called the \emph{commensurability class} of $M$.
Reid's question can be rephrased to ask: does the length spectrum of $M$ determine the commensurability class of $M$?

It is well-known that the length spectrum $\calL(M)$ does not determine the isometry class of $M$. In the setting of hyperbolic manifolds, the first counterexamples are due to Vign\'eras \cite{Vi}. Sunada \cite{Su} gave a general, group-theoretic method that allows one to start with a hyperbolic manifold $M_0$ and construct finite covers $M_1$ and $M_2$ that share the same length spectrum but are not isometric. All examples produced using Sunada's method, as well as the arithmetic examples produced by Vign\'eras, are commensurable by construction. This common feature was part of the motivation behind \refques{Reid}.

\subsection{Main results}
In this paper, we provide some evidence toward a negative answer to \refques{Reid}. We show that a large finite portion of the length spectrum $\calL(M)$ does not determine the commensurability class of $M$.

%%% Pagebreak ensures that the main theorem is not broken up between two pages. Remove if needed.
\pagebreak

\begin{theorem}\label{Thm:GeoSimManifolds}
For all sufficiently large $n \in \NN$, there exists a pair of non-isometric finite-volume hyperbolic $3$--manifolds $\left\lbrace N_{n}, N_{n}^{\mu}\right\rbrace$ such that:
\begin{enumerate}
\item\label{Itm:Volume} $\vol(N_{n})= \vol(N_{n}^{\mu})$, where this volume grows coarsely linearly with $n$.
\item\label{Itm:SameLengths} The (complex) length spectra of $N_{n}$ and $N_{n}^{\mu}$ agree up to length at least  $n$.
\item\label{Itm:ManyLengths} $N_{n}$ and $N_{n}^{\mu}$ have at least $e^{n} /n $ closed geodesics up to length $n$.
\item\label{Itm:Incommensurable} Each of $N_{n}$ and $N_{n}^{\mu}$ is the unique minimal orbifold in its commensurability class. In particular, $N_{n}$ and $N_{n}^{\mu}$ are incommensurable. 
	\end{enumerate}
\end{theorem}

The manifolds $N_{n}$ and $N_{n}^{\mu}$ can be taken to be either closed or one-cusped. The statement that volume grows \emph{coarsely linearly} means that there exist absolute constants $A, A' > 0$ such that
$An \leq \vol(N_{n})= \vol(N_{n}^{\mu}) \leq A' n$.

To place \refthm{GeoSimManifolds} in context, it helps to introduce the notions of commensurators and arithmeticity.
The commensurability class of a hyperbolic manifold $M = \HH^n / \Gamma$ is entirely determined by its (orientable) \emph{commensurator}, namely
\[
C^+(M) = C^+(\Gamma) = \{ g \in \Isom^+(\HH^n) : \Gamma \cap g \Gamma g^{-1} \text{ is finite index in both  $\Gamma$ and } g \Gamma g^{-1} \}.
\]
The commensurator $C^+(\Gamma)$ is always a group, but is not always discrete or torsion-free.
By a fundamental theorem of Margulis, its discreteness obeys a striking dichotomy.

\begin{theorem}[Margulis  \cite{Mar}]\label{Thm:MargulisCommensurator}
Let $M = \HH^n / \Gamma$ be a hyperbolic $n$--manifold of finite volume.
\begin{itemize}
\item If $M$ is arithmetic, the commensurator $C^+(\Gamma)$ is dense in $ \Isom^+(\HH^n)$. The commensurability class of $M$ contains infinitely many minimal elements.

\item If $M$ is non-arithmetic, the commensurator $C^+(\Gamma)$ is discrete in $ \Isom^+(\HH^n)$, hence $\calO = \HH^n / C^+(\Gamma)$ is the unique minimal orbifold in the commensurability class of $M$.
\end{itemize}
\end{theorem}

See Maclachlan and Reid \cite{MaRe} for the definition of an arithmetic manifold. 
%    For many applications, one may treat Margulis's characterization in \refthm{MargulisCommensurator} as a definition. 
%    In this paper, our only interaction with arithmeticity will come via \refthm{MargulisCommensurator}.
%    \marginpar{added sentence. Keep it?}
One way to interpret Margulis's theorem is that it is harder for non-arithmetic manifolds to be commensurable. Indeed, the manifolds $N_n$ and $N_n^\mu$  in \refthm{GeoSimManifolds} are non-arithmetic.

By contrast,  \refques{Reid} has a positive answer for arithmetic manifolds of dimension $n \neq 1 \!\!\mod 4$. In other words, the length spectrum of such a manifold determines its commensurability class.
This result is due to Reid \cite{Reid:isospectrality} in dimension $2$; to Chinburg, Hamilton, Long, and Reid \cite{CHLR:Commensurability} in dimension $3$; and to Prasad and Rapinchuk \cite{Prasad-Rapinchuk:WeakCommensurability, Prasad-Rapinchuk:survey} in all remaining dimensions satisfying $n \neq 1 \!\!\mod 4$.

%    In the setting of arithmetic hyperbolic manifolds of almost all dimensions, \refques{Reid} is known to have a positive answer.
%    
%    
%    By contrast, arithmetic manifolds display  different qualitative and quantitative behavior with regard to commensurability.
%    Reid showed  that \refques{Reid} has a positive answer if the $M_i$ are arithmetic surfaces \cite{Reid:isospectrality}. Chinburg, Hamilton, Long, and Reid gave a positive answer to \refques{Reid} if the $M_i$ are arithmetic $3$--manifolds \cite{CHLR:Commensurability}. Extending the above work, Prasad and Rapinchuk gave a positive answer to \refques{Reid} for arithmetic $n$--manifolds in all dimensions satisfying $n \neq 1 \mod 4$.

The contrast between arithmetic and non-arithmetic settings is further sharpened by the following counterpart to 
 \refthm{GeoSimManifolds}: a finite portion of the length spectrum of an arithmetic $3$--manifold \emph{does} determine its commensurability class.

\begin{theorem}[Linowitz--McReynolds--Pollack--Thomspon \cite{lmpt:effective-rigidity}]\label{Thm:LMPT}
There exists an absolute constant $c > 0$ such that the following holds. Let $M_1, M_2$ be closed arithmetic hyperbolic $3$--manifolds with volume less than $V$. If the length spectra  $\calL(M_1)$ and $\calL(M_2)$ agree for all lengths less than $ c \cdot \exp \left( \log V^{\log V} \right)$,  then $M_1$ and $M_2$ are commensurable.
\end{theorem}

Both \refthm{GeoSimManifolds} and \refthm{LMPT} have (easier) analogues in dimension $2$. See  \cite[Theorem 1.1]{lmpt:effective-rigidity} and \refobs{SimilarSurfaces} below.

%    It is worth observing that 
The length cutoffs in Theorems~\ref{Thm:GeoSimManifolds} and \ref{Thm:LMPT} are qualitatively different, as
the function $ (\log V)^{\log V}$ grows faster than any polynomial in $V$. Thus the length cutoff in \refthm{LMPT} is considerably higher than the cutoff of \refthm{GeoSimManifolds}, which is linear in the volume $V$.

%        Most of the prior work on \refques{Reid} has consisted of positive results in the setting of arithmetic hyperbolic  $2$-- and $3$--manifolds. \refthm{LMPT} is a particularly striking example in this vein. 

Outside the realm of hyperbolic manifolds, the analogue of  \refques{Reid} can have a negative answer for quotients of higher-rank symmetric spaces. 
For instance, 
Lubotzky, Samuels, and Vishne have constructed infinite families of closed arithmetic manifolds modeled on $PGL_n(\RR)/PO_n(\RR)$ for $n \geq 3$ that have the same length spectrum but are not commensurable \cite{LuSaVi}.

To our knowledge, the only prior  work on \refques{Reid} for non-arithmetic hyperbolic manifolds is due to Millichap \cite{Mi2}. He constructed a 
sequence of knots $K_n, K_n^\mu \subset S^3$ with incommensurable complements, such that their volumes grow linearly with $n$ and such that $S^3 \setminus K_n$ and $S^3 \setminus K_n^\mu$ share the same $n$ shortest geodesics. However, all of the geodesics in his construction have length uniformly bounded by $0.015$. Improving upon Millichap's result, we have the following analogue of \refthm{GeoSimManifolds} for knot complements in $S^3$.

\begin{theorem}
	\label{Thm:GeoSimKnots}
For each $n \gg 0$, there exists a pair of non-isometric hyperbolic knot exteriors $E_{n} = S^{3} \setminus K_{n}$ and  $E_{n}^{\mu} =  S^{3} \setminus K_{n}^{\mu}$ such that:
	
	\begin{enumerate}
\item $\vol(E_{n}) = \vol(E_{n}^{\mu})$, where this volume grows coarsely linearly with $n$. 
\item The (complex) length spectra of $E_{n}$ and $E_{n}^{\mu}$ agree up to  length at least  $2\log(n)$.
\item $E_{n}$ and $E_{n}^{\mu}$ have at least $n^2/(2 \log(n))$ closed geodesics up to length $2 \log(n)$.
\item $E_{n}$ is the unique minimal orbifold in its commensurability class, and the only knot complement in its commensurability class.
The same statement holds for $E_{n}^{\mu}$.
	\end{enumerate}
	
\end{theorem}

The knots $K_n$ in \refthm{GeoSimKnots} can be explicitly described via diagrams. See  \reffig{ArbKnots}.
%     for a visual summary.

The primary difference between \refthm{GeoSimManifolds} and \refthm{GeoSimKnots} is that in the latter, the length spectrum is only preserved up to a cutoff that grows logarithmically with volume. This contrasts with linear growth in \refthm{GeoSimManifolds} and super-polynomial growth in \refthm{LMPT}.
As a consequence, the number of geodesics that $E_n$ and $E_n^\mu$ share is also lower than the corresponding count for $N_n$ and $N_n^\mu$. Although both constructions involve a large product region (whose thickness is essentially $n$), closed geodesics in $E_{n}$ and $E_{n}^{\mu}$ can take shortcuts through the cusp, which results in a logarithmic penalty. 

On the other hand, the commensurability statement in \refthm{GeoSimKnots} is stronger than the one in \refthm{GeoSimManifolds}. We are able to conclude that $E_n$ and $E_n^\mu$ are the only knot complements in their commensurability classes, by relying on a theorem of Reid and Walsh \cite[Proposition 5.1]{ReWa} that is specific to hyperbolic knot complements in $S^3$.

%%%%%%%%%%%%
\subsection{Constructing spectrally similar manifolds}
To motivate the construction behind Theorems~\ref{Thm:GeoSimManifolds} and \ref{Thm:GeoSimKnots}, we first sketch a similar result in dimension $2$. 

\begin{observation}\label{Obs:SimilarSurfaces}
Let $S$ be an (orientable) surface of Euler characteristic $\chi(S) < -1$. Then, for all $n > 0$, $S$ admits a pair of complete hyperbolic structures $\Sigma_n$ and $\Sigma'_n$ such that:
\begin{enumerate}
\item\label{Itm:SameArea} $\area(\Sigma_n) = \area(\Sigma'_n) = -2\pi \chi(S)$.
\item\label{Itm:SurfaceSpectrum} The length spectra of $\Sigma_n$ and $\Sigma'_n$ agree up to length at least $n$. 
\item\label{Itm:SurfaceIncommens} $\Sigma_n$ and $\Sigma'_n$ are incommensurable.
\end{enumerate}   
\end{observation}

\begin{figure}[h]
	\begin{overpic}[scale=0.45]{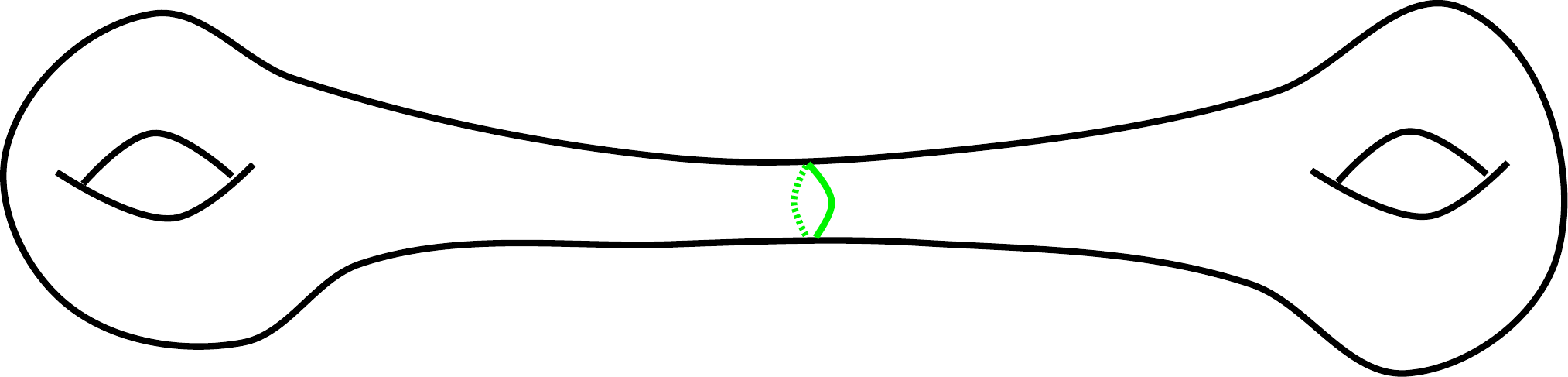}
		\put(50,3){\Large{$\gamma$}}
		\end{overpic}
	\caption{A closed genus two surface, whose hyperbolic structure $\Sigma_n$ has a large collar about a separating geodesic $\gamma$. To construct $\Sigma'_n$, we cut $\Sigma_n$ along $\gamma$, twist by some distance, and re-glue.}
	\label{Fig:SurfaceWithCollar}
\end{figure}

\refobs{SimilarSurfaces} can be proved as follows. The hypothesis $\chi(S) < -1$ ensures that $S$ contains an essential, separating, simple closed curve $\gamma$. Start with a hyperbolic metric $\Sigma_n$ in which $\gamma$ is short enough to have a collar of diameter $n/2$. By the collar lemma \cite{Buser:collar}, it suffices to have $\coth (\ell(\gamma)/2) > \cosh(n /4)$. Now, build $\Sigma'_n$ by cutting $S$ along $\gamma$, twisting some distance, and then regluing. See \reffig{SurfaceWithCollar}.

Since $\gamma$ is separating, and has a collar of diameter $n/2$, any closed geodesic shorter than $n$ must be disjoint from $\gamma$. All such geodesics are preserved when we cut and twist along $\gamma$, verifying \refitm{SurfaceSpectrum}. Twisting along $\gamma$ produces uncountably many distinct metrics on $S$, only finitely many of which can be commensurable with $\Sigma_n$. Thus a generic choice of twisting distance will also satisfy \refitm{SurfaceIncommens}.

Conclusion \refitm{SameArea}, a restatement of the Gauss--Bonnet theorem, is  included to illustrate the following point. One can get the length spectra of incommensurable surfaces $\Sigma_n$ and $\Sigma'_n$ to agree up to length $n$ while keeping the areas uniformly bounded. This differs from the setting of Theorems~\ref{Thm:GeoSimManifolds} and \ref{Thm:GeoSimKnots}, where our construction requires the volume to grow with the length cutoff. It also implies 
that for $n \gg 0$, at least one of $\Sigma_n$ or $\Sigma'_n$ must be non-arithmetic, as otherwise they would have to be commensurable by \cite[Theorem 1.1]{lmpt:effective-rigidity}.

Our $3$--dimensional construction requires a lot more machinery, but the intuitive idea of \refobs{SimilarSurfaces} still applies. First, we produce a $3$--manifold $N_n$ with a very thick collar about some separating surface $F_n$. The modern theory of Kleinian groups (see \refsec{Kleinian}) gives a number of ways to do this, the simplest of which is the following. We start with a standard pair of $3$--manifolds with boundary, denoted $T$ and $B$, 
% for \emph{top} and \emph{bottom} 
such that $\bdy T \cong \bdy B \cong S$.  Then, we glue these pieces via some large iterate $\varphi^{2n}$ of a pseudo-Anosov map $\varphi: S \to S$. A visual summary of the process appears in \reffig{BuildingNn}, and the details appear in \refsec{GeoSimManifolds}.
%    \refsec{Kleinian}. 
We show in \refprop{NnGeometry} that as $n \to \infty$, both $\vol(N_n)$ and the diameter of a collar about $F_n$ will grow linearly with $n$.

Once we have built a $3$--manifold $N_n$, with a thick collar about $F_n$, we obtain $N_n^\mu$ by a cut-and-paste operation, just as in \reffig{SurfaceWithCollar}.  However, by Mostow rigidity, most ways of cutting and regluing along $F_n$ will not preserve any of the fine-scale geometry of $N_n$. To avoid this problem, we cut $N_n$ along $F_n$ and reglue along a hyper-elliptic involution $\mu$, in a  procedure called \emph{mutation}. This forces $F_n$ to be  a closed surface of genus $2$ or one of several small surfaces with punctures; see \refdef{Mutation} for details. By a theorem of Ruberman \cite{Ru}, the mutation process can be realized as a rigid cut-and-paste operation along a minimal surface, preserving the geometry of $N_n$ on either side of this surface. In particular, $\vol(N_n) = \vol(N_n^\mu)$. Furthermore, a theorem of Millichap \cite[Proposition 4.4]{Mi2} says that any closed geodesic that can be \emph{homotoped} disjoint from $F_n$ will remain unperturbed under mutation.

To get our count of closed geodesics up to length $n$, we rely on the work of Huber \cite{Hu}, Margulis \cite{Mar2},     and Gangolli and Warner \cite{gangolli-warner}.
They showed that for all $L \gg 0$, a finite-volume hyperbolic $3$--manifold $M$ has approximately $e^{2L} / 2L$ geodesics up to length $L$. See \refsec{UniformCount} and Equation~\refeqn{MargulisCount}. However, this theorem is not uniform: the length cutoff $L_0$ after which the asymptotic estimate applies depends on the manifold $M$. In a manifold such as $N_n$, which has a ``bottleneck'' surface of small area, the Margulis asymptotic may not apply until a length cutoff much larger than $n$. On the other hand, \refthm{GeoSimManifolds} features a uniform count of closed geodesics that applies for all $n \gg 0$.

The solution to this dilemma is to count closed geodesics in 
$X_n$, the cover of $N_n$ (and $N_n^\mu$) corresponding to the surface $F_n$. These manifolds $X_n$ converge geometrically and algebraically to the infinite cyclic cover of the mapping torus $M_{\varphi}$, permitting a uniform count of their closed geodesics. See \refprop{XnGeodesics} for details. Once we have shown that $X_n$ has the desired number of geodesics, we argue that these geodesics project down to closed geodesics in $N_n$ and $N_n^\mu$ in a one-to-one fashion.  

The knot complements $E_n$ and $E_n^\mu$ from Theorem \ref{Thm:GeoSimKnots} are constructed in a very similar fashion, using nearly identical machinery. See \reffig{ArbKnots} for a summary and \refsec{ArbKnots} for all the details. As noted earlier, the main difference from \refthm{GeoSimManifolds} occurs in the cutoff length for the length spectra, which results from geodesics traveling through the cusp. 
% See Lemma \ref{Lem:KnotGeodesics}.

\subsection{Ruling out commensurability}
In \refobs{SimilarSurfaces}, a simple cardinality argument shows that a generic choice of twisting along $\gamma$ produces a surface incommensurable with $\Sigma_n$. Showing that a pair of $3$--manifolds is incommensurable typically requires stronger tools. This can be accomplished using algebraic invariants such as the invariant trace field or Bloch invariant (see e.g.\ Chesebro and DeBlois \cite{CheDeB}), or geometric invariants such as the canonical polyhedral decomposition (see Goodman, Heard, and Hodgson \cite{GoHeHo}). All of these invariants use the global geometry of the ambient manifold.

By contrast, Theorems~\ref{Thm:GeoSimManifolds} and \ref{Thm:GeoSimKnots} use a commensurability criterion that is much more local in nature. The criterion is based around thrice-punctured spheres, or \emph{pairs of pants}. By a theorem of Adams \cite{adams:pants}, every essential pair of pants in a cusped hyperbolic $3$--manifold is isotopic to a totally geodesic surface. Furthermore, all totally geodesic pairs of pants are isometric. This rigid geometry provides a lot of structure.

In the following theorem, a choice of cusp neighborhoods in a $3$--manifold $M$ induces a choice of cusp neighborhoods in a totally geodesic pair of pants $P \subset M$. We say that
$P$ is \emph{pairwise tangent} with respect to a horocusp $\calC \subset M$ if the three cusp neighborhoods of $P \cap \calC$ are tangent in pairs.
The notion of $P$ being \emph{geometrically isolated} is slightly harder to define; see \refdef{Isolated} for details. In addition, see \reffig{CapCuspView} for a visual description; the pairs of pants $P$ and $P'$ in that figure are each \emph{geometrically isolated on one side}.

\begin{theorem}\label{Thm:PantedCriterion}
Let $M$ be a finite-volume hyperbolic $3$--manifold with exactly three cusps. Let $C_1, C_2, C_3$ be embedded horospherical neighborhoods of the cusps of $M$. Suppose 
that $M$ contains exactly two pairs of pants $P$ and $P'$ that are pairwise tangent and geometrically isolated on one side, with respect to $C_1, C_2, C_3$. 
Suppose that each of $P$ and $P'$ meets every $C_i$, that $P$ and $P'$ are disjointly embedded, and that $P \cup P'$ cuts $N$ into a pair of submanifolds $M_+$ and $M_-$, where $M_+$ is asymmetric and $\vol(M_+) \neq \vol(M_-)$.

Let $s_i$ be a Dehn surgery coefficients on $\bdy C_i$. Then, for all choices of $s_i$ that are sufficiently long and sufficiently different, the filled manifold $M(s_1, s_2, s_3)$ is hyperbolic, non-arithmetic, and minimal in its commensurability class. This includes the case where $s_1 = \infty$, i.e.\ the cusp $C_1$ is left unfilled, and $s_2, s_3$ are sufficiently long and sufficiently different.
\end{theorem}

%    \marginpar{Emphasize that this works because non-arithmetic. Selling point: we only need local info about geometry.}

The meaning of \emph{sufficiently different} in the statement of the theorem is that the ratios between the lengths of $s_i$ are very large: $\ell(s_1) \gg \ell(s_2) \gg \ell(s_3)$. This meaning will be quantified and made more precise during the proof of the theorem in \refsec{MinimalManifold}.

Despite its somewhat complicated statement, \refthm{PantedCriterion} is quite powerful. To apply the theorem, one needs to verify that the local geometry of $M$ near two pairs of pants is as described. However, one needs no global information about $M$ beyond the asymmetry of $M_+$ and $M_-$ and the knowledge that all pairs of pants in $M$ have been located.

In our application, the manifolds $N_n, N_n^\mu$  in \refthm{GeoSimManifolds}  can each be described as a Dehn filling of some $3$--cusped manifold $M $ satisfying the necessary geometric conditions. See Figure \ref{Fig:BuildingNn} for a schematic. We can verify the hypotheses of \refthm{PantedCriterion} in infinitely many examples because the isometry class of $M_+$ will stay constant as $n$ varies, and the theorem requires very little information about $M_-$. This will allow us to conclude that each of $N_n, N_n^\mu, E_n$, and $E_n^\mu$ is minimal in its commensurability class.

The same type of local argument (involving a single horoball packing picture, in \reffig{TopTangleHoroballs}) works to show that $E_n$ and $E_n^\mu$ are the unique knot complements in their commensurability classes, establishing \refthm{GeoSimKnots}.

\subsection{Organization} In \refsec{PantsCommens}, we explain and prove the panted commensurability criterion of \refthm{PantedCriterion}. We also prove the closely related \refthm{PantsPreimage}, which describes the full preimage of a pair of pants under certain special covering maps.

In \refsec{Kleinian}, we review a number of notions and results from Kleinian groups. These results will be used to describe the large-scale geometry of certain submanifolds and covers of $N_n$ and $E_n$. In particular, they will be used to show that, just as in \reffig{SurfaceWithCollar}, each $N_n$ contains a surface $F_n$ with collar of thickness $n/2$.

In \refsec{UniformCount}, we review the work of Margulis \cite{Mar2} and Gangolli and Warner  \cite{gangolli-warner} on counting closed geodesics in hyperbolic manifolds. We also use the work of Pollicott and Sharp \cite{pollicott-sharp:error-terms} to obtain a uniform count of closed geodesics in a convergent sequence of surfaces (\refprop{BoundedSubsetTeich}) and quasifuchsian $3$--manifolds (\refprop{XnGeodesics}).

In \refsec{GeoSimManifolds}, we assemble these ingredients to prove \refthm{GeoSimManifolds}, following the outline described above. Finally, in \refsec{ArbKnots}, we describe the similar construction of spectrally similar knots, proving \refthm{GeoSimKnots}.

\subsection{Acknowledgements} We thank Alan Reid and Sam Taylor for a number of enlightening conversations.
We are also grateful to Ian Agol, Priyam Patel, Juan Souto, and Genevieve Walsh for their helpful comments and suggestions.

%%%%%%%%%%%%%%%%%%%%%%%%%%%%%%%%%%%%%%%%%%%%%%%%%%%%%%%%%%%

\section{A panted commensurability criterion}\label{Sec:PantsCommens}

This section develops conditions under which a  hyperbolic $3$--manifold $N$ is the minimal orbifold in its commensurability class. The main result,
Theorem \ref{Thm:PantedCriterion}, will be applied in Sections \ref{Sec:GeoSimManifolds} and \ref{Sec:ArbKnots}  to show that our spectrally similar manifolds are minimal in their respective commensurability classes, hence incommensurable with one another. In order to prove Theorem \ref{Thm:PantedCriterion}, we first prove Theorem \ref{Thm:PantsPreimage}, which describes how pairs of pants with particular geometric properties behave under covering maps.
Stating these theorems requires a handful of definitions.

\begin{definition}\label{Def:HoroballDiagram}
Let $M$ be a non-compact hyperbolic $3$--manifold with finite volume. Every non-compact end of $M$ is a \emph{cusp}, homeomorphic to $T^2 \times [0, \infty)$. Geometrically, each cusp is a quotient of a horoball in $\HH^3$ by a $\ZZ \times \ZZ$ group of deck transformations. We call this geometrically standard end a \emph{horospherical cusp neighborhood} or \emph{horocusp}.

A collection of horospherical cusp neighborhoods $\calC = \{C_1, \ldots, C_k\}$ is called \emph{maximal} if each $C_i$ is embedded in $M$, but no $C_i$ can be expanded further without overlapping itself or another horocusp. A maximal collection of horocusps in $M$ lifts to a collection of horoballs in $\widetilde{M} = \HH^3$, with each horoball tangent to some number of other horoballs. This is called a \emph{horoball packing} of $\HH^3$. See \reffig{Voronoi} for an example in dimension $2$.

In the upper half-space model of $\HH^3$, a horoball packing can be moved by isometry so that some horoball $H_\infty$ (covering some horocusp $C_i$ of $M$) consists of all points above Euclidean height $1$. 
%    In this case, the hyperbolic metric on $\bdy H_\infty \cong \RR^2$ coincides with the Euclidean metric. 
Every horoball tangent to $H_\infty$ is called \emph{full-sized}, and has Euclidean diameter $1$. Other horoballs will be smaller. The collection of all horoballs other than $H_\infty$, as projected vertically to $\bdy H_\infty = \RR^2$, is called a \emph{horoball diagram of $C_i$}. See \reffig{CapCuspView} for an example.
\end{definition}

%    \begin{definition}\label{Def:FordVoronoi}
%    Ford--Voronoi domain. To each horoball in $\HH^n$, assign the points closest to it. This gives the Ford--Voronoi domain. The dual is the canonical (Epstein--Penner) cell decomposition.
%    \end{definition}

\begin{definition}\label{Def:Isolated}
Let $M$ be a finite-volume hyperbolic $3$--manifold, equipped with a collection of embedded cusp neighborhoods $\calC = \{ C_1, \ldots, C_k\}$. Let $P \subset M$ be pair of pants. By a theorem of Adams \cite{adams:pants}, $P$ is totally geodesic. We say that $P$ is \emph{pairwise tangent} with respect to $\calC$ if the three cusp neighborhoods of $P \cap \calC$ are tangent in pairs.

If $P$ is pairwise tangent, every cusp neighborhood in $P$ will lift to a \emph{distinguished line} of pairwise tangent full-sized horoballs in the horoball diagram of the corresponding cusp $C_i \subset M$. This line has a particular slope, determined by the isotopy class of $P \cap \bdy C_i$. (See \reffig{CapCuspView}, where the slope of $P \cap \bdy C_i$ is $0$.) Observe that a choice of transverse orientation on $P$ determines  a transverse orientation on the line of full-sized horoballs.

We say that $P$ is \emph{geometrically isolated} if the full-sized horoballs corresponding to a component of $P \cap C_i$ are  only tangent to full-sized horoballs that lie on the distinguished line. Meanwhile, $P$ is \emph{geometrically isolated on one side} if the full-sized horoballs corresponding to a component of $P \cap C_i$ are only tangent to full-sized horoballs that lie on the distinguished line, or in a fixed transverse direction from the line. In other words, the full-sized horoballs on the distinguished line cannot be tangent to other full-sized horoballs on both sides of the distinguished line. For example, the pairs of pants $P$ and $P'$ depicted in \reffig{CapCuspView} are geometrically isolated to one side.
\end{definition}

\begin{theorem}\label{Thm:PantsPreimage}
Let $M$ be a cusped hyperbolic $3$--manifold, equipped with a collection of maximal horocusps. Let $P$ be a thrice punctured sphere that meets every cusp of $M$, which is pairwise tangent and geometrically isolated on one side. 

Consider a covering map $\psi$ from $M$ to an orbifold $\calO$, where the cusp neighborhoods of $M$ are equivariant with respect to the covering. Then the full preimage $\psi^{-1}\circ \psi(P)$ is a disjoint union of thrice punctured spheres, and furthermore any component $Q \subset \psi^{-1}\circ \psi(P)$ is also pairwise tangent and geometrically isolated on one side.
\end{theorem}

\begin{remark}\label{Rem:NonEquivariant}
The hypothesis that the cusps neighborhoods in $M$ are equivariant with respect to the covering projection $\psi: M \to \calO$ is somewhat restrictive. Without knowing $\calO$ in advance, it is hard to know whether a given choice of cusps in $M$ will be equivariant. In our application (\refthm{PantedCriterion}), we use a trick involving geodesic lengths to check that this hypothesis is satisfied.

There is an alternate version of \refthm{PantsPreimage} that does not require an equivariant choice of cusp neighborhoods. However, this alternate version of the theorem needs a stronger form of geometric isolation. Instead of considering full-sized horoballs that are tangent, as in \refdef{Isolated}, one needs to consider horoballs whose distance (along a lift of $\bdy C_i$) is no greater than $2\sqrt{2}$. If one such horoball lies on the distinguished line, then the other horoball must also lie on the distinguished line or in a fixed transverse direction. With this stronger hypothesis, one can relax the equivariance of cusps and still reach the same conclusion about the preimage $\psi^{-1}\circ \psi(P)$.

The proof of this alternate statement is considerably more involved, as it requires bounding the factor by which the cusps of $M$ must be resized in order to become equivariant. 
%    (By a theorem of Adams \cite{adams:waist2}, the factor is at most $2^{3/4}$.)
It also requires a detailed case-by-case analysis using the work of Adams on cusp areas of $3$--orbifolds with rigid cusps \cite{adams:small-volume-orbifolds}. We have decided to omit this alternate statement, as it is not needed for our application.
\end{remark}

\subsection{Proving \refthm{PantsPreimage}}

Before outlining the proof of \refthm{PantsPreimage}, we recall some facts about Euclidean $2$--orbifolds. 

\begin{definition}\label{Def:EucOrbifold}
An (orientable)  \emph{Euclidean $2$--orbifold} is a quotient $\RR^2 / \Gamma$, where $\Gamma $ is a discrete subgroup of $\Isom^+ (\RR^2)$.
By the Gauss--Bonnet theorem, such an orbifold must have Euler characteristic $0$. It follows that a compact  (orientable) Euclidean $2$--orbifold is either a torus, a pillowcase, or a turnover. Here, a  \emph{pillowcase} is a $2$--sphere  with 4 cone points of cone angle $\pi$. A \emph{turnover} is  a $2$--sphere with $3$ cone points of cone angles $2\pi/p$, $2\pi/q$, and $2\pi/r$, where the triple $(p,q,r)$ is one of $(2,3,6)$, $(2,4,4)$, or $(3,3,3)$. Euclidean turnovers are called \emph{rigid}, because they admit a unique Euclidean structure up to scaling. See e.g.\ \cite[Theorem 2.2.3]{chk:orbifold} for background.
\end{definition}

The proof of \refthm{PantsPreimage} contains the following steps. The first step, which involves ruling out rigid cusps in the quotient orbifold $\calO$, also plays a major role in studying commensurability of knot complements. See \cite[Proposition 9.1]{NeRe}, and compare \reflem{KnotUnique}.

% Also \cite[Lemma 3.2]{Wa}, and \cite{BBRW}.

\begin{enumerate}[1.]

\item Show that a cusp cross-section of $\calO$ cannot be rigid. In other words, the cross-section of every cusp of $\calO$ is a torus or pillowcase. 
We do this in \reflem{NotRigid}.

\item Show that $\psi(P)$ has no transverse self-intersections in $\calO$. We do this in \reflem{NoTransverse}.

\item It follows from \reflem{NotRigid} and \reflem{NoTransverse} that all the cusp neighborhoods of $\psi^{-1}\circ \psi(P)$ are parallel to those of $P$. We may reassemble them to form several pairs of pants.
\end{enumerate}

The following notation will be in use throughout the proof of \refthm{PantsPreimage}. Let $p_1, p_2, p_3$ be the three punctures of $P$. Then puncture $p_i$ is mapped to a maximal horocusp $C_i \subset M$. It may be the case that $C_i$ coincides with $C_j$ for some $i \neq j$, but we maintain the separate names. Let $D_i$ be the component of $P \cap C_i$ corresponding to puncture $p_i$. By hypothesis of \refthm{PantsPreimage}, the $2$--dimensional horocusps $D_1, D_2, D_3$ are disjointly embedded and tangent in pairs. This means
\[ \area(D_i) = \ell(\bdy D_i) = 2
\quad \text{ for all $i$}.
\]

\begin{lemma}\label{Lem:NotRigid}
Let $M$ be a cusped hyperbolic $3$--manifold. Suppose that $M$ has a collection of maximal cusps and a thrice-punctured sphere $P$  satisfying the hypotheses of \refthm{PantsPreimage}. 

Let $\psi: M \to \calO$ be a covering map, where the cusp neighborhoods of $M$ cover those of $\calO$. Then every cusp cross-section of $\calO$ is a torus or pillowcase. Furthermore, curves or arcs of the form $\psi(\bdy D_i)$ realize exactly one slope on each cusp of $\calO$.
\end{lemma}

Recall that a \emph{slope} on a torus $T^2$ is an isotopy class of essential simple closed curves. These isotopy classes on $T^2$ are in natural $1$-to-$1$ correspondence with $\QQ \cup \{ \infty \}$.
In a covering map $T^2 \to R$, where $R$ is a pillowcase, simple closed curves on $T^2$ project to closed curves or arcs between singular points on $R$. We say that curves or arcs on $R$ have the \emph{same slope} if they lift to closed curves on $T^2$ with the same slope. Equivalently, a closed curve and a curve or arc on $R$ have the same slope if they can be realized disjointly.

\begin{proof}[Proof of \reflem{NotRigid}]
Let $T_i = \bdy C_i$ be a cusp torus of $M$, and let $R_i = \psi(T_i)$. Observe that  $\widetilde{R_i} = \widetilde{T_i} = \RR^2$ contains a distinguished line of pairwise tangent full-sized horoballs, corresponding to $\widetilde{P}$. 
This line of horoballs projects to a particular slope on $T_i$, namely $\bdy D_i$.

 Suppose, for a contradiction, that $R_i$ is a (rigid) Euclidean turnover. 
 Then  $\pi_1(R_i)$ contains rotations of order $3$ or $4$. Thus each distinguished line of full-sized horoballs coming from $\widetilde{P}$ must intersect another line of full-sized horoballs at an angle of $2\pi/3$ or $\pi/2$. But the only way for two lines of pairwise tangent full-sized horoballs to intersect is if they share a horoball.
 This contradicts the hypothesis that $P$ is geometrically isolated on one side.

We have now established that $R_i$ is a torus or pillowcase, hence isotopy classes of curves and arcs on $R_i$ have well-defined slopes. It may very well happen that $\psi(T_i)$ and $\psi(T_j)$ are the same $2$--orbifold $R_i = R_j$.

Suppose, for a contradiction, that $\psi(\bdy D_i)$ and $\psi(\bdy D_j)$ represent distinct slopes on $R_i = R_j$. This means that $\psi(\bdy D_i)$ and $\psi(\bdy D_j)$ have a transverse intersection, possibly at a cone point of $R_i$. But then, lifting everything to the universal cover $\widetilde{R_i} = \RR^2$, we again find two different intersecting lines of full-sized horoballs. As above, this contradicts the the hypothesis that $P$ is geometrically isolated on one side.
\end{proof}

The following lemma will help us rule out self-intersections in $\psi(P)$.

\begin{lemma}\label{Lem:ClosestHoroball}
Let $M$ and $P$ be as in \refthm{PantsPreimage}. 
Let $D_1, D_2, D_3$ be pairwise tangent cusp neighborhoods in $P$.  Then any point $x \in P$ is within distance $\log (2/\sqrt{3})$ of some $D_i$. 

Furthermore, suppose that $\gamma$ is a geodesic path in $M$, of length $\ell(\gamma) \leq \log (2/\sqrt{3})$, whose origin is at $x \in P$ and whose endpoint runs perpendicularly into some torus $ \bdy C_j$. Then $\gamma \subset P$.
\end{lemma}

\begin{proof}
We begin by recalling the \emph{Ford--Voronoi domain} of $P$. 
For every $i$, let $E_i \subset P$ be the closure of the set of all points of $P$ that lie closer to $D_i$ than to any other $D_j$. Because the cusp neighborhoods $D_i$ are pairwise tangent, the collection $\{ D_i \}$ is invariant under all the symmetries of $P$. It follows that every $E_i$ is topologically a once-punctured disk, whose non-ideal boundary consists of two geodesic segments meeting at angles of $2\pi/3$. These vertices of angle $2\pi/3$ are the triple points where $E_1, E_2, E_3$ all meet in $P$. 
A computation in $\HH^2$, using \reffig{Voronoi}, shows that
\begin{equation}\label{Eqn:DistFromCusp}
\max_{x \in E_i} \, d(x, D_i) = \log \frac{2}{\sqrt{3}} = 0.1438 \ldots,
\end{equation}
with the maximum realized at the vertices. This proves the first assertion of the lemma.

\begin{figure}
\begin{overpic}{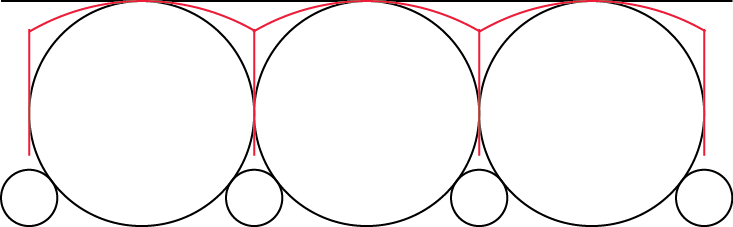}
\end{overpic}
\caption{The Ford--Voronoi domain for the pairwise tangent cusp neighborhoods in the pair of pants $P$. Horoballs are shown in black, while the boundaries of the Ford--Voronoi cells are in red. If the horoball about $\infty$ lies at Euclidean height $1$, the trivalent red vertices lie at height $\sqrt{3}/2$, which means the distance between them is $\log (2/\sqrt{3})$.}
\label{Fig:Voronoi}
\end{figure}

Now, suppose that $\gamma \subset M$ is a geodesic path as in the statement of the lemma. 
In the universal cover $\widetilde{M} = \HH^3$, there is a geodesic segment $\widetilde{\gamma}$ that starts at a point $\widetilde{x}$ covering $x$ and runs perpendicularly to a horoball $\widetilde{C_j}$. 
The geodesic $\widetilde{\gamma}$ is the unique shortest path from $\widetilde{x}$ to $\bdy \widetilde{C_j}$. Furthermore, $\widetilde{x}$ is contained in some Ford--Voronoi cell $\widetilde{E_i}$. Let $\widetilde{C_i}$ be the horoball containing the ideal point of $\widetilde{E_i}$; this will be one of the horoballs in $\HH^3$ closest to $\widetilde{x}$.

We may assume without loss of generality that $\widetilde{C_i}$ is a horoball about $\infty$ in the upper half-space model of $\HH^3$, positioned by isometry so that $\bdy \widetilde{C_i}$ lies at height $1$. This means that the copy of $\widetilde{P} \subset \HH^3$ containing $\widetilde{E_i}$  is a vertical plane, with $\widetilde{E_i}$ being the cell about $\infty$, as in \reffig{Voronoi}. 

If $\widetilde{C_i} = \widetilde{C_j}$, then the shortest path from $\widetilde{x}$ to $\bdy \widetilde{C_j}$ will be contained in $\widetilde{E_i}$, because $\widetilde{E_i}$ is totally geodesic. Thus $\widetilde{\gamma} \subset \widetilde{E_i}$, which implies $\gamma \subset P$.
Similarly, if $\widetilde{C_j}$ is a full-sized horoball that is tangent to $\widetilde{E_i}$, then the shortest path $\widetilde{\gamma}$ from $\widetilde{x}$ to $\widetilde{C_j}$ will be contained in the totally geodesic surface $\widetilde{P}$. 

Finally, suppose for a contradiction that  $\widetilde{C_j}$ is neither the horoball $\widetilde{C_i}$ about $\infty$, nor a
full-sized horoball that meets our chosen vertical copy of $\widetilde{P}$.
We may construct a piecewise geodesic path $\overline{\gamma}$ that follows a vertical geodesic from $\bdy \widetilde{C_i}$ to $\widetilde{x}$ and then follows $\widetilde{\gamma}$ from $\widetilde{x}$ to $\bdy \widetilde{C_j}$. The hypothesis on $\ell(\gamma)$, combined with \refeqn{DistFromCusp}, tells us that $\ell(\overline{\gamma}) \leq 
2 \log (2/\sqrt{3}) < 0.288$. Thus the Euclidean radius of $\widetilde{C_j}$ satisfies
\[
\rad(\widetilde{C_j}) = \frac{1}{2 \exp \big( d(\widetilde{C_i}, \widetilde{C_j} ) \big)} \geq \frac{1}{2 \exp \big( \ell(\overline{\gamma} ) \big)} \geq \frac{1}{2} \left( \frac{\sqrt{3}}{2} \right)^2 = \frac{3}{8}.
\]
As $\widetilde{C_j}$ must be disjoint from the full-sized horoballs that meet $\widetilde{P}$, a calculation using the Pythagorean theorem shows that the center of $\widetilde{C_j}$ must be at (Euclidean) distance at least $1/\sqrt{2}$ from the vertical copy of $\widetilde{P}$, with the lower bound on distance growing faster than $\rad(\widetilde{C_j})$ as the radius grows.
Since $\rad(\widetilde{C_j}) \geq 3/8$,  every point of $\widetilde{C_j}$ is at Euclidean distance at least $1/\sqrt{2} - 3/8 > 0.332$ from $\widetilde{P}$,
hence $\ell(\overline{\gamma}) > 0.332$. On the other hand, we have previously seen that $\ell(\overline{\gamma}) < 0.288$. This contradiction shows that $\widetilde{\gamma} \subset \widetilde{P}$.
\end{proof}

\begin{lemma}\label{Lem:NoTransverse}
Under the hypotheses of \refthm{PantsPreimage}, $\psi(P)$ has no transverse self-intersections. 
\end{lemma}

We remark that \reflem{NoTransverse} does not require $\psi \vert_P$ to be $1$-to-$1$. For instance, we may have a non-trivial covering map $\psi: P \to R$, where $R \subset \calO$ is a totally geodesic $2$--orbifold. See \reffig{PantsHyperelliptic} for an example of such a covering map.

\begin{proof}
We begin by noting that the $2$--dimensional cusp neighborhoods $D_i \subset P$ are disjointly embedded in $3$--dimensional cusp neighborhoods $C_i \subset M$. Furthermore, $\psi(C_i)$ and $\psi(C_j)$ are either disjoint or coincide. Hence the only way that $\psi(D_i)$ can have a transverse intersection with $\psi(D_j)$ is if $\psi(C_i) = \psi(C_j)$, and furthermore $\psi(\bdy D_i)$ and $\psi(\bdy D_j)$ represent distinct slopes. This possibility is ruled out by \reflem{NotRigid}.

Next, we will use the lack of transverse self-intersections between the cusp neighborhoods to rule out transverse self-intersections elsewhere in $\psi(P)$.
Recall the Ford--Voronoi decomposition of $P$ into cells $E_1, E_2, E_3$, as in \reflem{ClosestHoroball}, and suppose for a contradiction that $\psi(P)$ has a transverse self-intersection. Then, for some pair $(i,j)$, a point $x_i \in E_i$ has the same image as $x_j \in E_j$. Furthermore, in $P$ there are disjoint neighborhoods $U_i$ of $x_i$ and $U_j$ of $x_j$ such that $\psi(U_i)$ meets $\psi(U_j)$ at a nonzero angle.

Let $\gamma \subset \psi(E_i)$  be the shortest path in $\calO$ from $\psi(x_i)$ to $\psi(D_i)$. By \refeqn{DistFromCusp}, this path has length $\ell(\gamma) \leq \log (2 /\sqrt{3})$. We consider two different lifts of $\gamma$ to $M$.
Let $\gamma_i \subset M$ be the lift of $\gamma$ starting at $x_i$; this lift is contained in $E_i$ by construction. But since $\psi(x_i) =  \psi(x_j)$, there is also a lift $\gamma_j$ of $\gamma$ that starts at $x_j$. Since the cusps of $M$ cover cusps of $\calO$, each of  $\gamma_i$ and $\gamma_j$ terminates on the boundary of some cusp, hitting the cusp perpendicularly. Thus, by \reflem{ClosestHoroball}, each of $\gamma_i$ and $\gamma_j$ lies in $P$.

It follows that $\psi(P)$ has a transverse self-intersection along all of $\gamma = \psi(\gamma_i) = \psi(\gamma_j)$. In particular, $\psi(P)$ has a transverse self-intersection at the endpoint of $\gamma$, which will have to be contained in $\psi(\bdy D_i) \cap \psi(\bdy D_j)$. But we have already checked that $\psi(D_i)$ has no transverse intersections with $\psi(D_j)$.
\end{proof}

\begin{proof}[Proof of \refthm{PantsPreimage}]
Since $\psi(P)$ has no transverse self-intersections by \reflem{NoTransverse}, we know that $\psi^{-1} \circ \psi(P)$ is a disjoint union of embedded, totally geodesic surfaces. We need to check that each component is a pair of pants.

Let $Q$ be a component of $\psi^{-1} \circ \psi(P)$. Then the horoball packing and Ford--Voronoi decomposition of $\widetilde{Q}$ are isometric to that of $\widetilde{P}$, as depicted in \reffig{Voronoi}. 
%In particular, exactly $3/\pi$ of the area of $Q$ is contained in the cusp neighborhods, just as in $P$.

For each cusp $C_i$ of $M$, \reflem{NotRigid} implies that each component of $Q \cap \bdy C_i$ has the same slope as $\bdy D_i$. Thus each component of $Q \cap C_i$ is an isometric copy of $D_i$, with area exactly $2$. Since the density of the cusp  neighborhoods in $Q$ is the same as in $P$, namely $3/\pi$, each Ford--Voronoi cell of $Q$ meeting $C_i$ will have to be an isometric copy of $E_i$, with area $2\pi/3$ and two vertices of angle $2\pi/3$.

The only way to glue together several copies of $E_1$ and obtain a complete, connected surface is if three different copies are glued at every vertex. Thus $Q$ is built from three copies of $E_1$ glued by isometry, hence $Q$ is a pair of pants. In addition, $Q$ is pairwise tangent and geometrically isolated on one side, because horoballs corresponding to $P \cap C_i$ will pull back to horoballs corresponding to $Q \cap C_i$.
\end{proof}

\subsection{Minimal manifolds in a commensurability class} \label{Sec:MinimalManifold}
Our main application of \refthm{PantsPreimage} is \refthm{PantedCriterion}. Before restating and proving the theorem, we need a definition.

\begin{definition}\label{Def:NormalizedLength}
Let $M$ be a cusped hyperbolic $3$--manifold, equipped with a horocusp $C$. For a slope $s$ on $\bdy C$, define the \emph{normalized length} of $s$ to be
\[
\hat{L}(s) = \ell(s)/ \sqrt{\area(\bdy C)},
\]
where $\ell(s)$ is the length of a Euclidean geodesic in the isotopy class of $s$. The quantity $\hat{L}(s)$ is scale--invariant, hence does not depend on the choice of cusp neighborhood.
\end{definition}

\begin{named}{\refthm{PantedCriterion}}
Let $M$ be a finite-volume hyperbolic $3$--manifold with exactly three cusps. Let $C_1, C_2, C_3$ be embedded horospherical neighborhoods of the cusps of $M$. Suppose 
that $M$ contains exactly two pairs of pants $P$ and $P'$ that are pairwise tangent and geometrically isolated on one side with respect to $C_1, C_2, C_3$. 
Suppose that each of $P$ and $P'$ meets every $C_i$, that $P$ and $P'$ are disjointly embedded, and that $P \cup P'$ cuts $N$ into a pair of submanifolds $M_+$ and $M_-$, where $M_+$ is asymmetric and $\vol(M_+) \neq \vol(M_-)$.

Let $s_i$ be a Dehn surgery coefficients on $\bdy C_i$. Then, for all choices of $s_i$ that are sufficiently long and sufficiently different, the filled manifold $M(s_1, s_2, s_3)$ is hyperbolic, non-arithmetic, and minimal in its commensurability class. This includes the case where $s_1 = \infty$, i.e.\ the cusp $C_1$ is left unfilled, and $s_2, s_3$ are sufficiently long and sufficiently different.
\end{named}

\begin{proof}
If each Dehn surgery slope $s_i \subset \bdy C_i$ is sufficiently long, the filled $3$--manifold $M(s_1, s_2, s_3)$ will be hyperbolic. 
By Gromov's theorem \cite[Theorem 6.5.6]{thurston:notes}, its volume is less than $\vol(M)$.
Since there are only finitely many arithmetic hyperbolic manifolds of bounded volume (see Borel \cite{Borel:volume}), it follows that for sufficiently long $s_i$, the filled manifold $N = M(s_1, s_2, s_3)$ will be non-arithmetic.

For each $i$, let $\hat{L}_i = \ell(s_i)/ \sqrt{\area(\bdy C_i)}$ denote the normalized length of $s_i$. When each $s_i$ is sufficiently long,  the core of the $i$-th solid torus will be isotopic to a geodesic $\gamma_i$. Neumann and Zagier \cite[Proposition 4.3]{NeumannZagier} showed that the length of this geodesic can be expressed as
\begin{equation}\label{Eqn:NeumannZagier}
\ell(\gamma_i) = \frac{2\pi }{ \hat{L}_i^{2} } + O \left( \frac{1}{\hat{L}_i^{4}} \right).
\end{equation}
See also Hodgson and Kerckhoff \cite[Theorem 5.12]{HK:shape} and Magid \cite[Theorem 1.2(ii)]{Magid:deformation} for explicit estimates on $\ell(\gamma_i)$.

Let $\calO_{\min}$ denote the orientable hyperbolic $3$--orbifold of minimal volume. Gehring, Marshall, and Martin \cite{gehring-martin:minimal-orbifold, marshall-martin:minimal-orbifold2} have identified this orbifold and showed that  $\vol(\calO_{\min})  \approx 0.03905$. We set $V = \vol(M) / \vol(\calO_{\min})$. The precise meaning of ``sufficiently different'' slope lengths is that (in some permutation of the indices) we have $(\hat{L}_1)^2 \gg V (\hat{L}_2)^2 \gg V^2 (\hat{L}_3)^2$. By \refeqn{NeumannZagier} and our choice of sufficiently long $s_i$ (to overwhelm the $\hat{L}_i^{-4}$ error term), this implies
\begin{equation}\label{Eqn:UnequalLengths}
\ell(\gamma_3) > V \ell(\gamma_2) > V^2 \ell(\gamma_1).
\end{equation}
One final application of the ``sufficiently long'' hypothesis ensures that $\gamma_1, \gamma_2, \gamma_3$ are the three shortest closed geodesics in $N$, by a factor of at least $V$.

Since $N$ is non-arithmetic, \refthm{MargulisCommensurator} says there is a unique minimal orbifold $\calQ$ in the commensurability class of $N$. Since $\vol(\calQ) \geq \vol(\calO_{\min})$ and $\vol(N) < \vol(M)$, the degree of the covering map $\psi: N \to \calQ$ is bounded above by $V$. By \refeqn{UnequalLengths}, this means every $\gamma_i$ is the complete preimage of its image, hence $\psi$ restricts to a covering map $\psi \vert_M: M \to \calO$, where $\calO$ is a hyperbolic $3$--orbifold homeomorphic to $\calQ \setminus \cup_i \psi(\gamma_i)$.

From here on, the proof works by restricting the degree of $\psi$ further and further, until we conclude that $\deg(\psi) = 1$, hence $N = \calQ$.

Consider what $\psi$ does to $P$. In the incomplete metric on $M$ whose completion is $N$, the three boundary components of $P$ are mapped to powers of $\gamma_1, \gamma_2, \gamma_3$. We have already checked that $\psi$ maps these three closed geodesics to three distinct geodesics in $\calQ$. Thus, in the restricted map $\psi \vert_M: M \to \calO$, the three boundary components of $P$ are sent to three distinct cusps of $\calO$. This means that $\psi \vert_P$ is either $1$-to-$1$ or a quotient by the hyper-elliptic involution; see Figure \ref{Fig:PantsHyperelliptic}. In particular,  $\psi \vert_P$ is at most a $2$-to-$1$ map, and similarly for $\psi \vert_{P'}$.

Since the three cusps of $M$ project to distinct cusps of $\calO$, the cusp neighborhoods $C_1, C_2, C_3$ are automatically equivariant with respect to the cover $\psi:M \to \calO$. Thus \refthm{PantsPreimage} applies. 
%    \marginpar{checked that \refthm{PantsPreimage} applies}

\begin{figure}
\begin{overpic}{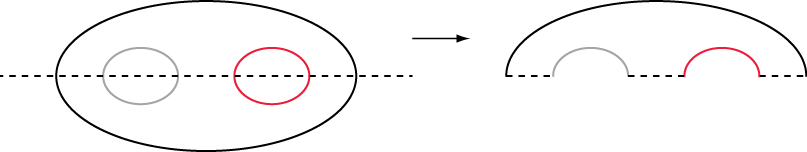}
\put(24,14){$P$}
\put(53.5,15){$\psi$}
\put(80,14){$T$}
\end{overpic}
\caption{A quotient map $\psi: P \to T$, where $P$ is a pair of pants and $T$ is a $2$--orbifold. If the three boundary components of $P$ are sent to distinct boundary components of $T$, then either $\psi$ is $1$-to-$1$  (this case is not shown), or $T$ is an ideal triangle with mirrored boundary along the dashed arcs shown in the figure.}
\label{Fig:PantsHyperelliptic}
\end{figure}

Let $R = \psi(P \cup P') \subset \calO$. Then  we have
\[
\psi^{-1}(R) \: = \:  \left( \psi^{-1} \circ \psi(P \cup P') \right)
\: = \: \left( \psi^{-1}  \psi(P) \cup \psi^{-1}  \psi(P') \right)
\: \subset \:  P \cup P'.
\]
The final inclusion comes from \refthm{PantsPreimage}, combined with the hypothesis that $P$ and $P'$ are the only  pairwise-tangent, geometrically isolated pairs of pants in $M$.
Since the opposite inclusion $P \cup P' \subset \psi^{-1}(R)$ is immediate, we conclude that $P \cup P' = \psi^{-1}(R)$.
Thus a point of $R$ has at most $4$ preimages in $M$: at most $2$ in $P$ and at most $2$ in $P'$. Since $\psi: M \to \calO$ is orientation--preserving, the singular locus of $\calO$ is at most $1$--dimensional, hence some points of $R = \psi(P \cup P')$ must be non-singular. We conclude that $\deg(\psi) \leq 4$.

% Now, observe that $\psi$ restricts to a covering space $M_+ \cup M_- \to \calO \setminus R$. 
% We claim that $\calO \setminus R$ has two connected components, namely $\psi(M_+)$ and $\psi(M_-)$. To see this, l
Let $\eta(P \cup P')$ be a regular neighborhood of $P \cup P'$, chosen to be equivariant with respect to $\psi$. Then $\bdy \eta(P \cup P') \cap M_+ = P_+ \cup P_+'$, where $P_+$ and $P_+'$ are pairs of pants parallel to $P$ and $P'$, respectively. Similarly, $\bdy \eta(P \cup P') \cap M_- = P_- \cup P'_-$.
Since the neighborhood was chosen to be equivariant,   $\eta(P \cup P')$ maps to a regular neighborhood $\eta(R)$.
% whose boundary is $\psi(P_\pm \cup P'_\pm)$. 
 Note that if  $\psi(P)$ is a pair of pants, then $\psi(P_+)$ will be a pair of pants parallel to $\psi(P)$. Meanwhile, if $\psi(P) = T$ is an ideal triangle with mirrored boundary, as in \reffig{PantsHyperelliptic}, then $\psi(P_+)$ will still be a pair of pants that forms the entire boundary of $\eta(T)$. 
 Thus $\psi \vert_{P_+}$ is $1$-to-$1$, hence $\psi \vert_{P_+ \cup P_+'}$ has degree at most $2$, and this degree is the same as that of $\psi \vert_{P_- \cup P'_-}$.

Now, consider the restriction of $\psi$ to $M_+$ and $M_-$. Since $\bdy (M_+ \setminus \eta(P \cup P')) = (P_+ \cup P_+')$, and similarly for $M_-$, we have an equality of degrees:
\begin{equation}\label{Eqn:DegreeEquality}
\deg \left( \psi \vert_{M_+} \right) = \deg \big( \psi \vert_{P_+ \cup P_+'} \big) =
\deg \big( \psi \vert_{P_- \cup P'_-} \big) = \deg \left( \psi \vert_{M_-} \right)
\: \in \: \{ 1, 2 \}.
\end{equation}

We learn two things from \refeqn{DegreeEquality}. First, 
 $\calO \setminus R$ must have two connected components, namely $\psi(M_+)$ and $\psi(M_-)$. 
Otherwise, if $\calO \setminus R$ is connected, each of $M_+$ and $M_-$ would have to cover it. But then the equality of degrees
$\deg ( \psi \vert_{M_+} ) = \deg ( \psi \vert_{M_-} )$ would imply $\vol(M_+) = \vol(M_-)$, contradicting our hypotheses. 
Thus $\psi(M_+)$ and $\psi(M_-)$ are disjoint subsets of $\calO$, hence
\[
\deg \left( \psi   \right)
= \deg \left( \psi \vert_M  \right)
= \deg \left( \psi \vert_{M_+} \right)
\: \in \: \{ 1, 2 \}.
\]
Second, observe that a covering map of degree $2$ must be regular, hence given by a symmetry of $M$. Any non-trivial covering transformation of $M$ must respect $P \cup P'$, hence must either interchange the components $M_\pm$ or stabilize $M_\pm$. But $M_+$ cannot be interchanged with $M_-$ because they have different volumes, and a symmetry cannot stabilize $M_+$ because we have assumed $M_+$ is asymmetric. It follows that $\deg(\psi) = 1$, hence 
 $N = \calQ$ is minimal in its commensurability class.
 
 We close the proof by observing that the entire argument goes through if one cusp of $M$, say $C_1$, is left unfilled. If we choose long Dehn filling slopes $s_2$ and $s_3$ such that $(\hat{L}_2)^2 \gg V (\hat{L}_3)^2$, equation \refeqn{UnequalLengths} would still hold with the convention that $\ell(\gamma_1) = 0$. Thus any covering map $\psi: N \to \calQ$ would restrict to a covering map $\psi: M \to \calO$, and the rest of the argument applies verbatim.
\end{proof}

\section{The geometry of pared convex cores}\label{Sec:Kleinian}

Most of this section is a review of standard notions and results in the theory of Kleinian groups, particularly Kleinian surface groups. In \refprop{LinearGeometryGrowth}, we describe a sequence of geometrically finite surface groups whose convex core boundaries are separated by distance approximately $n$. This construction will be used to build the spectrally similar manifolds in \refsec{GeoSimManifolds} and the spectrally similar knot complements in \refsec{ArbKnots}.

None of the results in this section are original. Our line of exposition is mostly modeled on that of Brock and Dunfield \cite{BrDu}, as their construction has many similarities to ours. Another excellent reference for this material is Canary, Epstein, and Green \cite{CaEpGr}.

\subsection{Kleinian surface groups and convergence}\label{Sec:AH(S)}
Let $S = S_{g,k}$ be a compact (orientable) surface of genus $g$ with $k$ boundary components. We define the complexity $\xi(S) = 3g+k-3$, and require that  $\xi(S) > 0$. 
%The only hyperbolic surface with $\xi(S) = 0$ is a pair of pants. 
A representation $\rho: \pi_1(S) \to PSL(2,\CC)$ is called \emph{type-preserving} if $\rho$ maps peripheral loops in $S$ to parabolic elements. Define the space
\[
AH(S) = \{ \rho: \pi_1(S) \to PSL(2,\CC) : \rho \text{ discrete, faithful, type-preserving} \}/\sim,
\]
where the equivalence relation $\sim$ is conjugation of the image in $PSL(2,\CC)$. 
The image $\rho(\pi_1 S)$ is called a \emph{Kleinian surface group}.
Every element of $AH(S)$ naturally corresponds to a hyperbolic $3$--manifold $M = \HH^3 / \rho(\pi_1 S)$, homeomorphic to $S \times \RR$ and equipped with a homotopy equivalence to $S$. This homotopy equivalence $S \to S \times \{ * \} \subset M$, which is well-defined up to homotopy, is called a \emph{marking}. 

A hyperbolic $3$--manifold $M = \HH^3 / \Gamma$ is determined up to isometry by the conjugacy class of $\Gamma$ in $PSL(2,\CC)$. One way to pin down a particular conjugacy representative is to fix a \emph{baseframe} $\omega$ in $M$ (that is, a basepoint $x \in M$ together with an orthonormal frame in $T_x M$), and require that $\omega$ must lift to a fixed baseframe $\widetilde{\omega}$ at the origin in $\HH^3$. Thus pairs $(M, \omega)$ consisting of a hyperbolic $3$--manifold and a baseframe are in $1$-to-$1$ correspondence with Kleinian groups $\Gamma \subset PSL(2,\CC)$, whereas un-framed manifolds are in $1$-to-$1$ correspondence with conjugacy classes.

We will sometimes abuse notation by keeping baseframes (or equivalently, conjugacy representatives) implicit. Thus we will refer to either $\rho : \pi_1(S) \to PSL(2,\CC)$ or $M = \HH^3 / \rho(\pi_1 S)$ as elements of $AH(S)$. 

The topology on $AH(S)$ is that of \emph{algebraic convergence}. In this topology, a sequence of representations $\rho_n$ converges to $\rho_\infty$ if, for a generating set $\gamma_1, \ldots, \gamma_k$ of $\pi_1(S)$, the sequence of matrices $\rho_n(\gamma_i) \in PSL(2,\CC)$ converges to $\rho_\infty(\gamma_i) \in PSL(2, \CC)$ for every $i$.

There is a finer topology on the space of (framed) hyperbolic manifolds, called the  \emph{Chabauty topology}. In this topology, a sequence of Kleinian groups $\Gamma_n \subset PSL(2,\CC)$ \emph{converges geometrically} to $\Gamma \subset PSL(2,\CC)$ if the following hold:

\begin{enumerate}[\quad (a)]
\item For each $\gamma \in \Gamma$, there are $\gamma_n \in \Gamma_n$ so that $\gamma_n \to \gamma$.

\item If a sequence $\{\gamma_n \in \Gamma_n \}$ converges in $PSL(2,\CC)$, then the limit lies in $\Gamma$.
\end{enumerate}
The Chabauty topology is metrizable \cite[Proposition 3.1.2]{CaEpGr}. See Biringer \cite{Biringer:Chabauty} for an explicit metric.

In a sequence of representations $\rho_n : \pi_1 S \to PSL(2, \CC)$, with image $\Gamma_n = \rho_n (\pi_1 S)$, algebraic convergence only requires the $\rho_n$--images of the generators to converge, while geometric convergence requires the images of all group elements to converge. A sequence $\Gamma_n$ converges \emph{strongly} if it converges both algebraically and geometrically, and furthermore the algebraic and geometric limits coincide. Equivalently, a sequence of framed manifolds $(M_n, \omega_n)$ converges strongly to $(M_\infty, \omega_\infty)$ if the algebraic and geometric limits coincide.

There is an intrinsic description of geometric convergence, as follows.  Let $M_n$ be a sequence of complete hyperbolic manifolds, each equipped with a baseframe $\omega_n$. Let $M_\infty$ be another manifold with a baseframe $\omega_\infty$. For every $R >0$, let 
$B_R(x_\infty) \subset M_\infty$ be the metric $R$--ball about the basepoint $x_\infty$ of $\omega_\infty$. We say that $(M_\infty, \omega_\infty)$ is the \emph{Gromov--Hausdorff limit} of  $(M_n, \omega_n)$ if, for every $R > 0$, we have embeddings
\begin{equation}\label{Eqn:GeometricLimit}
        \psi_{n,R} : (B_R(x_\infty), \, \omega_\infty) \hookrightarrow (M_n, \, \omega_n),
 \end{equation}
     for all $n$ sufficiently large, which converge to isometries in the $C^\infty$ topology as $n \to \infty$. 
     See \cite[Theorem E.1.13]{BP} or \cite[Theorem 3.2.9]{CaEpGr} for the equivalence between Gromov--Hausdorff convergence and convergence in the Chabauty topology.

%    
%    Let $X$ and $Y$ be compact metric spaces. The \emph{Gromov--Hausdorff distance} between $X$ and $Y$ is defined to be
%    \[
%    d_{GH}(X,Y) = \inf \: d_{X \sqcup Y} (X, Y) ,
%    \]
%    where the distance on the right-hand side is Hausdorff distance, and the infimum is taken over all metrics on the disjoint union $X \sqcup Y$ whose restrictions to $X$ and $Y$ recover the given metrics on $X$ and $Y$. Informally, $d_{GH}(X,Y)$ will be close to $0$ whenever $X$ and $Y$ are nearly isometric.
%    
%    Let $M_n$ be a sequence of complete hyperbolic manifolds, each equipped with a basepoint $x_n$, and 
%    let $M_\infty$ be another manifold with a basepoint $x_\infty$. For every $R >0$, let $B_R(x_n) \subset M_n$ 
%    be the metric $R$--ball about $x_n$, and similarly for $B_R(x_\infty) \subset M_\infty$ The pairs $(M_n, x_n)$ are said to \emph{converge geometrically} to $(M_\infty, x_\infty)$ if, for every $R > 0$, we have
%    \[
%    % \lim_{n \to \infty} d_{GH}(B_R(x_n), B_R(x_\infty)) = 0.
%       \psi_{n,R} : (B_R, x_\infty) \longrightarrow (M_n, x_n),
%    \]
%     for all $n$ sufficiently large, which converge to isometries in the $C^\infty$ topology as $n \to \infty$.
%    
%    A sequence $M_n \in AH(S)$ converges \emph{strongly} to $M_\infty = \HH^3 / \rho_\infty(\pi_1 S)$ if, in addition to geometric convergence, the associated representations $\rho_n$ converge to $\rho_\infty$.

\subsection{Limit sets and cores}\label{Sec:LimitSet}

Given a hyperbolic $3$--manifold $N = \HH^3 / \Gamma$, the \emph{limit set} $\Lambda(\Gamma) \subset \bdy \HH^3$ is the set of accumulation points of any orbit $\Gamma x$. The \emph{convex core}, denoted $\core(N)$, is the smallest geodesically convex subset of $N$. Alternately, $\core(N)$ is the  quotient by $\Gamma$ of the convex hull of the limit set $\Lambda(\Gamma)$.
If $\core(N)$ has finite volume, $N$ is called \emph{geometrically finite}.

Let $C^{\epsilon}(N) \subset N^{< \epsilon}$ be
the disjoint union of $\epsilon$--thin horospherical neighborhoods about the cusps of $N$ (if any), for $\epsilon$ equal to the Margulis constant. Then  the \emph{pared convex core} is defined to be $\core^0(N) = \core(N) \setminus C^{\epsilon}(N)$.

%    There are two special kinds of geometrically finite manifolds in $AH(S)$.
The set $QF(S) \subset AH(S)$ consists of the representations $\rho$ for which $\Lambda(\rho (\pi_1 S))$ is a Jordan curve. In this case, both $\rho \in QF(S)$ and the associated manifold $M = \HH^3 / \rho(\pi_1 S)$ are called \emph{quasifuchsian}. By a theorem of Bers, $QF(S)  \cong \Teich(S) \times \Teich(S)$, where $\Teich(S)$ is the Teichm\"uller space of $S$ and the two coordinates correspond to conformal structures on the two ends of $M \cong S \times \RR$.

\subsection{Pared acylindrical manifolds}\label{Sec:Pared}

We will study certain geometrically finite manifolds in $AH(S)$ whose hyperbolic structures correspond to special points of $\bdy \Teich(S) \times \bdy \Teich(S)$.

Let $Q$ be a \emph{pants decomposition} of $S$, namely a collection of disjointly embedded, essential simple closed curves cutting $S$ into pairs of pants; then $|Q| = \xi(S) > 0$. A pants decomposition $Q$ defines a point of $\bdy \Teich(S)$, in which the lengths of curves in $Q$ have been pinched all the way to $0$; see Brock \cite[Page 502]{brock:quasifuchsian} for more detail.
% Let $Q'$ be another such collection, where no simple closed curve is isotopic into both $Q$ and $Q'$. 
A pair of pants decompositions $(Q, Q')$ determines the pair
\[
M_S(Q,Q') = \left( S \times [ 0,1] , \:\: Q \times \{ 0 \} \cup Q' \times \{ 1 \} \cup \bdy S \times \{ \tfrac{1}{2} \} \right).
\]
This $3$--manifold is homeomorphic to $S \times [ 0,1]$, with a designated \emph{paring locus} along $Q \times \{ 0 \} \cup Q' \times \{ 1 \} \cup S \times \{ \tfrac{1}{2} \} $.
If no simple closed curve is isotopic into both $Q$ and $Q'$, the resulting $3$--manifold will be \emph{pared acylindrical}, meaning that it does not contain any essential annuli with boundary in the paring locus.

By Thurston's hyperbolization theorem \cite{thurston:survey}, every pared acylindrical $3$--manifold admits a finite-volume hyperbolic metric, with rank--$1$ cusps along every closed curve in the paring locus, and totally geodesic boundary along every pair of pants. Mostow--Prasad rigidity applied to the double of $M_S(Q,Q')$ verifies that this hyperbolic metric is unique up to isometry. 
 The inclusion  $S \to S \times \{ \tfrac{1}{2} \} \subset M_S(Q,Q')$ is a homotopy equivalence, which specifies a marking of $M_S(Q,Q')$ by the surface $S$. Thus $M_S(Q,Q') \in AH(S)$.

The pared acylindrical manifold $M_S(Q,Q')$ constructed above is the convex core of a complete hyperbolic manifold $\hat{M}_S(Q,Q')$, which contains a flaring end for each totally geodesic pair of pants in $\bdy M_S(Q, Q')$. The pared convex core, $\core^0 M_S(Q,Q')$, is a compact $3$--manifold whose boundary is the union of horospherical annuli along the paring locus (including the annuli of $\bdy S \times [0,1]$) and the $\epsilon$--thick subsurfaces of the pants. As those pants live either in $S \times \{0\}$ or $S \times \{1 \}$, we get a natural decomposition of the non-parabolic portion of $\bdy \core^0 M_S(Q,Q')$ into \emph{lower} and \emph{upper} surfaces, denoted $\bdy_- \core^0 M_S(Q,Q')$ and $\bdy_+ \core^0 M_S(Q,Q')$. The lower surface can be identified as
\[
\bdy_- \core^0 M_S(Q,Q') = \bdy  \core^0 M_S(Q,Q') \cap (S \times \{ 0 \}),
\]
and similarly for the upper surface.

\subsection{Pseudo-Anosov double limits}\label{Sec:DoubleLimit}

Given any pseudo-Anosov homeomorphism $\varphi : S \to S$, we can construct the mapping torus $M_{\varphi} = S \times \left[ 0,1\right]  / (x, 1) \sim (\varphi(x), 0)$. By a theorem of Thurston \cite{thurston:fibered-manifolds}, $M_{\varphi}$ has a unique hyperbolic metric. Define $\widetilde{M}_{\varphi}$ to be the infinite cyclic cover of $M_{\varphi}$ corresponding to $\pi_{1}(S)$. Then  $S \times \mathbb{R} \cong \widetilde{M}_{\varphi} \in AH(S)$. 

For a fixed pants decomposition $Q$ on $S$, we consider the pared manifold
\[
W_{\varphi,n} = M_S(\varphi^{-n}Q, \, \varphi^n Q).
\]
For all $n$ sufficiently large, all the curves of $\varphi^{-n}(Q)$ will be distinct from those of $\varphi^n( Q)$, hence $W_{\varphi,n}$ is pared acylindrical. See \cite[Figure 3.7]{BrDu} for a helpful visual depiction of $M_S(\varphi^{-n}Q, \, \varphi^n Q)$, as well as  $\widetilde{M}_{\varphi}$.

\begin{proposition}%[Thurston] 
	\label{Prop:DoubleLimit}
For every baseframe $\omega_\infty$ on $\widetilde{M}_{\varphi}$, there is a sequence of baseframes $\omega_n$ on $W_{\varphi,n} = M_{S}(\varphi^{-n}Q, \varphi^{n}Q)$, such that $(W_{\varphi,n}, \omega_n)$ converges strongly to $(\widetilde{M}_{\varphi}, \omega_\infty)$. 
%    Furthermore, the frames $\omega_n$ are based at points $x_n$ such that $\lim_{n \to \infty} d(x_n, \bdy W_{\varphi, n}) \to \infty$.
\end{proposition}

This result is due to Thurston \cite{thurston:fibered-manifolds}, and forms the key step in his proof that $M_{\varphi}$ is hyperbolic. 
We refer to \cite[Theorem 1.2]{BrBr} for an alternate proof, and to
 \cite[Proposition 3.4]{BrDu} for the version stated here.

We also learn that, in two different senses, the geometry of $W_{\varphi,n}$ grows linearly with $n$.

\begin{proposition}	\label{Prop:LinearGeometryGrowth}
There are positive constants $A_\varphi, A'_\varphi, B_\varphi, B'_\varphi$, depending on $\varphi$, such that the following holds. For all $n \gg 0$,
\begin{equation}\label{Eqn:LinearVolume}
A_\varphi \leq \frac{\vol(W_{\varphi,n})}{n} \leq A'_\varphi
\end{equation}
and
\begin{equation}\label{Eqn:LinearDist}
B_\varphi \leq \frac{d_0 (\bdy_- \core^0 W_{\varphi,n}, \, \bdy_+ \core^0 W_{\varphi,n}) }{n} \leq B'_\varphi 
\end{equation}
% Furthermore, the injectivity radius of $\core^0 W_{\varphi,n}$ is bounded below by $\delta_\varphi$.
Here, $d_0(\cdot, \cdot)$ is the shortest length of a path from the lower boundary of $\core^0 W_{\varphi,n}$ to the upper boundary of $\core^0 W_{\varphi,n}$, among paths that remain inside $\core^0 W_{\varphi,n}$.
\end{proposition}

\begin{proof}
This follows as a consequence of theorems that relate the geometry of $W_{\varphi,n}$ to the coarse geometry of certain combinatorial graphs associated to the surface $S$.

 The \emph{curve graph}, denoted $\Curve(S)$, has vertices corresponding to isotopy classes of essential simple closed curves on $S$ and edges corresponding to curves that can be realized disjointly. (When $\xi(S) = 1$, the definition must be modified slightly. Edges in $\Curve(S_{1,1})$ correspond to curves that intersect once, while edges in $\Curve(S_{0,4})$ correspond to curves that intersect twice.)
The \emph{pants graph}, denoted $\Pants(S)$, has vertices corresponding to isotopy classes of pants decompositions and edges corresponding to moves where a curve is removed, liberating a copy of $S_{1,1}$ or $S_{0,4}$ inside $S$, and replaced by another curve that intersects it a minimal number of times on that subsurface. See \cite[Figure 3]{brock:quasifuchsian}.

Each of $\Curve(S)$ and $\Pants(S)$ is endowed with the \emph{graph metric}, in which every edge has length $1$. The mapping class group $MCG(S)$ acts by isometries on these metric graphs. For any $\varphi \in MCG(S)$, there is a well-defined \emph{stable translation length}
\begin{equation}\label{Eqn:StableLength}
t_\Curve(\varphi) = \lim_{n \to \infty} \frac{d_{\Curve(S)}(v, \varphi^n(v))}{n}, 
\qquad 
t_\Pants(\varphi) = \lim_{n \to \infty} \frac{d_{\Pants(S)}(v, \varphi^n(v))}{n}, 
\end{equation}
where the limit is independent of the base vertex $v$. If $\varphi$ is a pseudo-Anosov element, we have $t_\Pants(\varphi) > 0$ and $t_\Curve(\varphi) > 0$; for $\Curve(S)$, the converse also holds.

With this background, Equation~\refeqn{LinearVolume} follows from a theorem of Brock \cite{brock:quasifuchsian}. He showed that there are positive constants $K$ and $K'$, depending only on $S$, such that
\[
K \, d_{\Pants(S)}(\varphi^{-n}Q, \varphi^n Q) - K \leq \vol(W_{\varphi,n}) \leq K' \, d_{\Pants(S)}(\varphi^{-n}Q, \varphi^n Q).
\]
(To interpret Brock's theorem in our setting, one needs to regard the pants decomposition $Q$ as a point of the Weil--Petersson completion of $\Teich(S)$, denoted $\overline{\Teich(S)}$, which allows us to consider the pair $(\varphi^{-n}Q, \varphi^n Q) \in \overline{\Teich(S)} \times \overline{\Teich(S)}$. 
%    Then, observe that \cite[Theorems 1.1 and 1.2]{brock:quasifuchsian} apply to geometricaly finite $3$--manifolds defined by points of $\overline{\Teich(S)} \times \overline{\Teich(S)}$; 
See \cite[Pages 499 and 502]{brock:quasifuchsian}.)
After dividing by $n$ and taking a limit as in \refeqn{StableLength}, we obtain  \refeqn{LinearVolume} with  any values $A_\varphi$ and $A'_\varphi$ satisfying $0< A_\varphi < 2K t_\Pants(\varphi) $ and $A'_\varphi > 2K' t_\Pants(\varphi)$.

Working toward \refeqn{LinearDist}, recall that $\bdy \core^0 W_{\varphi,n}$ is the union of $\epsilon$--thick subsurfaces of the pants in $S \times \{0, 1\}$, plus horospherical annuli that correspond to the boundary of the $\epsilon$--thin part about the paring locus. Here, as above, $\epsilon$ is the Margulis constant. Let $\gamma$ be a component of $Q$, and let $T^{\epsilon} (\varphi^n \gamma)$ be the horospherical annulus of $\bdy \core^0 W_{\varphi,n}$ corresponding to the thin part about $\varphi^n ( \gamma)$. Similarly, let $T^{\epsilon} (\varphi^{-n} \gamma)$ be the horospherical annulus corresponding to $\varphi^{-n} ( \gamma)$.

With this setup, the lower bound of \refeqn{LinearDist} follows from a theorem proved independently by Bowditch \cite[Theorem 5.4]{bowditch:ELC} and Brock and Bromberg \cite[Theorem 7.16]{BrBr}. They show that there exists a positive constant $E$, depending only on $S$, such that 
\begin{equation}\label{Eqn:LowerElectricDist}
E \, d_{\Curve(S)}( \varphi^{-n} \gamma, \:  \varphi^{n} \gamma) - E  \leq d_0 (T^{\epsilon} (\varphi^{-n} \gamma), \: T^{\epsilon} (\varphi^{n} \gamma)) .
\end{equation}
When $S$ is not a closed surface, this inequality crucially relies on measuring only the lengths of paths that stay inside the pared convex core, $\core^0 W_{\varphi, n}$. See the remark following  \cite[Theorem 7.16]{BrBr}. Since $\bdy_+ \core^0 W_{\varphi, n}$ has bounded diameter (this fact also requires the core to be pared), every point of $\bdy_+ \core^0 W_{\varphi, n}$ lies at uniformly bounded distance from $T^{\epsilon} (\varphi^n \gamma)$. Thus, for some constant $E' > E$, we have
\[
E \, d_{\Curve(S)}( \varphi^{-n} \gamma, \:  \varphi^{n} \gamma) - E'  \leq d_0 (\bdy_- \core^0 W_{\varphi,n}, \, \bdy_+ \core^0 W_{\varphi,n})  .
\]
Compare \cite[Corollary 7.18]{BrBr} for a very similar statement. Now, dividing by $n$ and taking a limit as in \refeqn{StableLength}, we obtain the lower bound of \refeqn{LinearDist} with any $ B_\varphi \in (0, 2E t_\Curve(\varphi) )$.

The upper bound of  \refeqn{LinearDist} also follows by considering $d_{\Curve(S)}( \varphi^{-n} \gamma, \:  \varphi^n \gamma)$. A straighforward argument using the bounded diameter lemma for surfaces (see e.g.\ Biringer and Souto \cite[Theorem 4.1]{biringer-souto}) shows that there is a positive function $F(\delta)$ such that if the injectivity radius of $\core^0 W_{\varphi,n}$ is bounded below by $\delta$, then 
\begin{equation}\label{Eqn:UpperElectricDist}
d_0 (\bdy_- \core^0 W_{\varphi,n}, \, \bdy_+ \core^0 W_{\varphi,n})
 \leq F(\delta) \, d_{\Curve(S)}( \varphi^{-n} \gamma, \:  \varphi^{n} \gamma) + F(\delta).
\end{equation}
But Brock and Dunfield show in \cite[Theorem 3.5]{BrDu} that for $n \gg 0$, 
 $\core^0 W_{\varphi,n}$ satisfies
\[
\injrad \left( \core^0 W_{\varphi,n} \right) \geq \injrad \left(\core^0 M_\varphi \right) / 2 > 0,
\]
everywhere outside a bounded--diameter collar of $\bdy_{\pm} \core^0 W_{\varphi,n}$ whose geometry converges with $n$. Since this bounded collar must itself have injectivity radius bounded below, it follows that there is a value $\delta = \delta_\varphi$ that serves as a lower bound on $\injrad \left( \core^0 W_{\varphi,n} \right)$ for all $n \gg 0$. 
Now, dividing \refeqn{UpperElectricDist} by $n$ and taking a limit as in \refeqn{StableLength} yields the upper bound of \refeqn{LinearDist}.
\end{proof}

\begin{remark}\label{Rem:CoarseConstants}
We make two obervations about the statement and proof of \refprop{LinearGeometryGrowth}. First, the upper bound of \refeqn{LinearDist} depends rather inefficiently on $\varphi$. One could obtain a much tighter estimate by using the so-called \emph{electric distance} between the upper and lower boundary of $\core^0 W_{\varphi,n}$, as in Bowditch \cite[Theorem 5.4]{bowditch:ELC} and Biringer--Souto \cite[Theorem 4.1]{biringer-souto}. Since the upper bound of \refeqn{LinearDist} will not be needed in the sequel, and is included mainly for completeness, we chose the formulation that is easiest to state.

Second, the lower bound of Equation~\refeqn{LinearDist} will be our main point of entry for the cutoff up to which the length spectra of $N_n$ and $N_n^\mu$ agree in \refthm{GeoSimManifolds}; see \refprop{NnGeometry} and \reflem{SameGeodesicsMutation}. 
To make this cutoff as simple as possible, it will be convenient to choose a pseudo-Anosov $\varphi$ for which $2 B_\varphi > 1$. As the above proof illustrates, it suffices to choose a $\varphi$ for which $4 E t_{\Curve}(\varphi) > 1$, where $E$ comes from Equation~\refeqn{LowerElectricDist} and $t_{\Curve}(\varphi)$ is as in \refeqn{StableLength}. For instance, by  \refeqn{StableLength}, a sufficiently high power of any pseudo-Anosov suffices.
\end{remark}

%%%%%%%%%%%%%%%%%%%%%%%%%%%%%%%%%%%%%%%%%%%%%%%%%%

\section{Counting geodesics, uniformly}\label{Sec:UniformCount}

In the statement of \refthm{GeoSimManifolds}, we claim that the manifolds $N_n$ and $N_n^\mu$ contain a certain number of closed geodesics up to length $n$. For any finite--volume manifold $M = \HH^d / \Gamma$, let $\calL(M) =  (\ell_1, \ell_2, \ldots )$ be the length spectrum of $M$, and define
\[
\pi_M(L) = \max \{ i : \ell_i \leq L \} =  \big( \text{the number of closed geodesics in $M$ of length} \leq L \big).
\]
The study of the asymptotic behavior of $\pi_M(L)$ has a long and distinguished history, starting with Huber \cite{Hu} and Margulis \cite{Mar2}. They proved that for any closed hyperbolic manifold $M$ of dimension $d$,
\begin{equation}\label{Eqn:MargulisCount}
\pi_M(L) \sim \frac{e^{h L}}{h L},
% \quad \text{as} \quad L \to \infty,
\quad \text{where} \quad h = d-1.
\end{equation}
Here and below, the notation $f(x) \sim g(x)$ means that $\lim_{x \to \infty} f(x) / g(x) = 1$. 
In fact, Margulis worked out a version of \refeqn{MargulisCount} in variable negative curvature for an appropriate definition of $h$. Gangolli and Warner showed that \refeqn{MargulisCount} also holds for non-compact hyperbolic $d$--manifolds of finite volume \cite{gangolli-warner}. Among many other works in the subject, Roblin \cite{roblin} proved an analogous result for geometrically finite hyperbolic $3$--manifolds such as $M_S(Q,Q')$; see \refrem{Underestimate}.

Although equation \refeqn{MargulisCount} is beautiful in its simplicity, we cannot use it directly because the rate of convergence depends on the manifold $M$. On the other hand, we are looking for a uniform statement that will hold in all the manifolds $N_n$ or all the $W_{\varphi,n}$, for $n \gg 0$. To obtain the desired result, we start with a uniform statement in bounded regions of Teichm\"uller space.

The Techm\"uller space $\Teich(S)$ can be interpreted as the space of discrete, faithful, type-preserving representations $\rho: \pi_1(S) \to PSL(2,\RR)$, up to conjugation in $PSL(2,\RR)$. 
%    Peripheral elements in $\pi_1(S)$ are required to be represented by parabolics. 
Closed geodesics in $S$ are in $1$-to-$1$ correspondence with non-peripheral conjugacy classes in $\pi_1(S)$. Thus, given a hyperbolic structure $\Sigma \in \Teich(S)$, every conjugacy class $[\gamma]$ has an associated $\Sigma$--length $\ell(\gamma) = \ell_\Sigma(\gamma)$, which can be computed via
\[
\cosh \frac{\ell(\gamma) }{ 2}  = \frac{ \operatorname{tr} \rho_\Sigma(\gamma) }{2},
\]
where $\rho_\Sigma$ is the representation corresponding to $\Sigma$. This makes it clear that $\ell_\Sigma(\gamma)$ is a conjugacy invariant, which varies continuously with $\Sigma$. The gist of the following proposition is that  $\pi_\Sigma(L)$ also varies 
in a well-behaved fashion over $\Teich(S)$.

\begin{proposition}\label{Prop:BoundedSubsetTeich}
Fix a topological surface $S = S_{g,k}$, and let $R \subset \Teich(S)$ be a compact region of the Teichm\"uller space of $S$. Then there is a constant $L_0$ depending on $R$, such that for every hyperbolic structure $\Sigma \in R$ and for every $L \geq L_0$, we have
\begin{equation}\label{Eqn:UniformCount}
\pi_\Sigma(L) \geq  \frac{e^{L}}{L}.
\end{equation}
\end{proposition}

\begin{proof}
This follows from the work of Pollicott and Sharp \cite{pollicott-sharp:error-terms}, who worked out the error term in the asymptotic formula \refeqn{MargulisCount}. They did this by studying the zeta function
\[
\zeta(z) = \zeta_\Sigma(z) = \prod_{[\gamma]} \left( 1 - e^{-z \ell(\gamma)} \right)^{-1},
\]
where $z \in \CC$ and the product ranges over all the non-peripheral conjugacy classes in $\pi_1(\Sigma)$. By \cite[Proposition 5]{pollicott-sharp:error-terms}, there is a constant $c_0(\Sigma) < 1$ such that $\zeta_\Sigma(z)$ converges to an analytic function on the half-plane $\mathop{Re}(z) > c_0(\Sigma)$, except for a simple pole at $z=1$. In this region of convergence (which is uniform on compact sets), the analytic function $\zeta_\Sigma(z)$ depends continuously on $\Sigma$, hence $c_0(\Sigma)$ also depends continuously on $\Sigma$.

The continuous dependence on $\Sigma$ means that on a compact set $R \subset \Teich(S)$, we may take a uniform constant 
\[
c_0 = c_0(R) = \max \{ c_0(\Sigma) : \Sigma \in R \} < 1,
\] 
such that $\zeta_\Sigma(z)$ converges on $\mathop{Re}(z) > c_0(R)$ for every $\Sigma \in R$,
except for a simple pole at $z=1$. Making $c_0(R)$ larger if necessary, we may assume $c_0(R) \geq 0$.

Pollicott and Sharp then show that for every $c = c(\Sigma)  \in \big( \frac{c_0(\Sigma) + 1}{2}, 1 \big)$
the number of closed geodesics up to length $L$ satisfies
\begin{equation*}%\label{Eqn:Pollicott-Sharp}
\pi_{\Sigma}(L) = \li(e^L) + O(e^{cL}).
\end{equation*}
Here, $\li(y)$ is the logarithmic integral
\begin{equation}\label{Eqn:li}
\li(y) = \int_2^y \frac{du}{\log u},
\qquad
\text{which satisfies} \quad
\li(y) \sim \frac{y}{\log y} \quad \text{as} \quad y \to \infty.
\end{equation}
See \cite[Proposition 6 and page 1033]{pollicott-sharp:error-terms} for the definition of $c(\Sigma)$ in terms of $c_0(\Sigma)$. 

Since there is a uniform value $c_0(R) < 1$ that works for every $\Sigma \in R$, we may also take a unform value $c = c(R)  \in \big( \frac{c_0(R) + 1}{2}, 1 \big)$. The proofs of \cite[Propositions 6 and 7]{pollicott-sharp:error-terms} show that the multiplicative constant implicit in $O(\cdot)$ can also be taken uniformly for every $\Sigma \in R$. Thus, for positive constants $c = c(R) < 1$ and $A = A(R)$, we have
\begin{equation}\label{Eqn:Pollicott-Sharp}
\pi_{\Sigma}(L) \geq \li(e^L) - Ae^{cL} 
\qquad \forall \: \Sigma \in R.
\end{equation}

It remains to derive \refeqn{UniformCount} from \refeqn{Pollicott-Sharp}. We do this via the following

\begin{claim}\label{Claim:Calculus}
For every $A \in \RR$ and every $c < 1$, we have $\displaystyle{ \lim_{x \to \infty} \left( \li(e^x) - Ae^{cx} - \tfrac{e^x}{x} \right) = \infty}$.
\end{claim}

To prove the claim, we compute the derivative, using \refeqn{li}:
\[
\frac{d}{dx} \left( \li(e^x) - Ae^{cx} - \frac{e^x}{x}  \right) \: = \: \left( \frac{1}{\log(e^x) } \cdot e^x - Ac \, e^{cx} \right) - \left( \frac{e^x}{x}   - \frac{e^x}{x^2} \right) \: = \:  - Ac \, e^{cx} + \frac{e^x}{x^2} \, ,
\]
which blows up as $x \to \infty$ whenever $c < 1$. Since the derivative is eventually very large, the function $  \left( \li(e^x) - Ae^{cx} - \tfrac{e^x}{x} \right) $ will also blow up as $x \to \infty$.

Combining \refclaim{Calculus} and \refeqn{Pollicott-Sharp} shows that there is a number $L_0 = L_0(R) > 0$, such that 
\[
\pi_{\Sigma}(L) \: \geq \: \li(e^L) - Ae^{cL} \: \geq \: \frac{e^L}{L}
\qquad \text{for} \qquad 
\Sigma \in R, \quad L \geq L_0 . \qedhere
\]
\end{proof}

\begin{remark}\label{Rem:moduli}
In \refprop{BoundedSubsetTeich}, it suffices to let $R \subset \Teich(S)$ be any subset whose image in the moduli space $\mathcal{M}(S) = \Teich(S) / MCG(S)$ is compact. For instance, one could let $R$ be the entire \emph{thick part} of $\Teich(S)$, consisting of all (marked) hyperbolic structures whose shortest geodesic has length at least $\epsilon > 0$. Since markings of $S$ play no role in counting closed geodesics up to length $L$, compactness is only needed in the image of $R$ in $\mathcal{M}(S)$.
\end{remark}

The next result uses \refprop{BoundedSubsetTeich} to show that geometrically finite manifolds such as $W_{\varphi,n}$ contain the appropriate number of closed geodesics shorter than $L$. We move up to the $3$--dimensional setting using the machinery of pleated surfaces, which we now recall.

\begin{definition}\label{Def:PleatedSurf}
A \emph{lamination} $\lambda$ on a surface $S$ is a $1$--dimensional foliation of a closed subset of $S$. Every point of $\lambda$ lies on a unique $1$--dimensional leaf. The lamination $\lambda$ is \emph{finite} if it has finitely many leaves, and  \emph{filling} if every complementary region of $S \setminus \lambda$ is an ideal triangle. If $S$ is a closed surface, every finite lamination contains at least one closed leaf, together with finitely many lines asymptotic toward the closed leaves. If $S$ has punctures, one example of a finite, filling lamination is an ideal triangulation.

Given a lamination $\lambda$ on $S$, and a hyperbolic $3$--manifold $M$, a \emph{pleating map realizing $\lambda$} is a proper map $f:S \to M$, such that every leaf of $\lambda$ is mapped to a geodesic and every component of $S \setminus \lambda$ is mapped to a (possibly self-intersecting) totally geodesic surface. The image $f(S)$ is called a \emph{pleated surface}. For $M \in AH(S)$, Thurston showed every finite, filling lamination $\lambda$ is realized by a pleating map in the homotopy class of the marking, provided that no closed leaf of $\lambda$ is homotopic to a parabolic in $M$. See \cite[Proposition 9.7.1]{thurston:notes} and \cite[Theorem 5.3.6]{CaEpGr}.
\end{definition}

\begin{proposition}\label{Prop:XnGeodesics}
Fix a hyperbolic surface $S$ with complexity $\xi(S) > 0$.
Let $X_n \in AH(S)$ be a sequence of marked, geometrically finite hyperbolic $3$--manifolds converging strongly to $X_\infty \in AH(S)$. 
Then there is a constant $L_0$ depending only on  $X_\infty$, and a constant $n_0$ depending on the sequence $X_n$, 
such that for all $L \geq L_0$ and $n \geq n_0$,
\[
\pi_{X_n}(L) \geq  \frac{e^{L}}{L}.
\]
\end{proposition}

\begin{proof}

Fix a finite, compact, filling lamination $\lambda$ on $S$, such that the marking map $S \to X_\infty$ does not map any closed leaf of $\lambda$ to a parabolic.
    Thurston's realization theorem (\cite[Proposition 9.7.1]{thurston:notes} and \cite[Theorem 5.3.6]{CaEpGr}) says that there is a pleating map  $f_{\lambda,\infty}: S \to X_\infty$ realizing $\lambda$. Choose a basepoint $x_\infty \in X_\infty$ that lies in a totally geodesic triangle in the image of $f_{\lambda, \infty}$, together with an orthonormal frame $\omega_\infty$ at $x_\infty$. 

Since $X_n$ converges to $X_\infty$, a closed leaf of $\lambda$ can only be parabolic in finitely many of the $X_n$. Thus, for $n \gg 0$, there is also a  pleating map $f_{\lambda,n}: S \to X_n$ realizing $\lambda$. By pulling back the path-metric from $X_n$ via $f_{\lambda,n}$, each one of these pleating maps endows $S$ with a complete (marked) hyperbolic metric, denoted $\Sigma_n$. 

 Since $X_n$ converges geometrically to  $X_\infty$, equation~\refeqn{GeometricLimit} says that the compact set $f_{\lambda,\infty}(S) \cap \core^0 X_\infty$ lies in a metric ball about $x_\infty$ that is almost--isometric to a region of $X_n$ for every large $n$.     In particular, the $1$--dimensional set $f_{\lambda,\infty}(\lambda)$, which is a union of geodesics in the metric on $X_\infty$, must be nearly geodesic in the metric on $X_n$ for $n$ large. As $n \to \infty$, this union of finitely many almost-geodesic curves can be perturbed to be geodesic in $X_n$ by a smaller and smaller homotopy.
      
    By the strong convergence of $X_n$, the marked hyperbolic structures $\Sigma_n$ converge in $\Teich(S)$ to a hyperbolic structure $\Sigma_\infty$, which the pleating map $f_{\lambda,\infty}$ induces on $S$. (Compare \cite[Theorem 5.2.2]{CaEpGr} and \cite[Lemma 1.3]{Ohshika:limits}.) Thus we may choose an $\epsilon$--ball $R_\epsilon$ about $\Sigma_\infty \in \Teich(S)$ (say, in the Teichm\"uller metric) such that  $\Sigma_n \in R_\epsilon$ for all $n \geq n_0$.

Since each pleating map $f_{\lambda,n}: S \to X_n$ is $1$--Lipschitz with respect to the metric $\Sigma_n$, every closed geodesic on $\Sigma_n$ tightens to a shorter geodesic in $X_n$. Thus we may apply \refprop{BoundedSubsetTeich} to every $\Sigma_n \in R_\epsilon$ and learn that there is a length cutoff $L_0$ such that for  all $L \geq L_0$ and $n \geq n_0$,
\[
\pi_{X_n}(L) \geq \pi_{\Sigma_n}(L) \geq  \frac{e^{L}}{L}. \qedhere
\]
\end{proof}

\begin{remark}\label{Rem:Underestimate}
The reason why \refprop{XnGeodesics} assumes geometric finiteness is that geometrically infinite manifolds in $AH(S)$ actually contain an \emph{infinite} set of closed geodesics shorter than some $L_0 = L_0(S)$.

Even for geometrically finite manifolds, the statement of \refprop{XnGeodesics} is surely an underestimate. Roblin \cite{roblin} showed that a geometrically finite hyperbolic $3$--manifold $M = H^3 / \Gamma$ satisfies
\[
\pi_M(L) \sim \frac{e^{hL}}{hL},
\]
where $h = h(M)$ is the Hausdorff dimension of the limit set $\Lambda(\Gamma)$. By a theorem of Bowen \cite{bowen:quasi-circles}, if $M$ is not Fuchsian, then $h(M) > 1$. Thus $\pi_M(L)$ grows strictly faster than $e^L / L$.

However, because the rate of convergence of $\pi_M(L)$ to $e^{hL} / {hL}$ may a priori depend on $M$, Roblin's result does not seem to imply the uniform statement that we need.
\end{remark}

\section{Spectrally similar $3$--manifolds}\label{Sec:GeoSimManifolds}

%%%%%%%%%%%%%%%%%%%%%%%%%%%%%%%%%%%%%%%%%%%%%%%%%%%%%%%%%%%%%%%%%%%%%%%%%%%

This section is devoted to proving \refthm{GeoSimManifolds}. To that end, \refsec{Construct} describes how to build the spectrally similar hyperbolic $3$-manifolds $N_{n}$ and $N_{n}^{\mu}$. Each $N_{n}$ comes from performing sufficiently long Dehn fillings on a manifold $J_{n}$ that is constructed by gluing a pair of \emph{caps}, denoted $T$ and $B$, to the ends of a large product region $W_{\varphi, n}$. In \refsec{Mutation}, we describe a cut-and-paste operation known as \emph{mutation}. Mutating $N_{n}$ produces its partner $N_{n}^{\mu}$, which shares much of the geometry of $N_n$. 
We will show that for $n \gg 0$, both $N_{n}$ and $N_{n}^{\mu}$ have the properties that are claimed in Theorem \ref{Thm:GeoSimManifolds}. 

\subsection{A general recipe}
We will use certain \emph{caps} with the following properties.

\begin{definition}\label{Def:Caps}
A compact, orientable $3$--manifold $M$  is called
\begin{enumerate}
\item\label{Itm:Simple} \emph{simple} if $M$ does not contain any essential spheres, tori, disks, or annuli.
\item\label{Itm:Asymmetric} \emph{asymmetric} if every self-homeomorphism of $M$ is isotopic to the identity.
\end{enumerate} 
We say that $M$ is a \emph{simple, asymmetric cap} for a surface $S$ if $\bdy M \cong S$ and both properties \refitm{Simple} and \refitm{Asymmetric} hold for $M$.
\end{definition}

By Thurston's hyperbolization theorem \cite{thurston:survey}, a $3$--manifold with boundary of genus at least $2$ is simple if and only if it admits a hyperbolic metric with totally geodesic boundary. 

We extend \refdef{Caps} to pared $3$--manifolds, as in \refsec{Pared}, by designating a collection of curves $Q \subset \bdy M$ as the paring locus. In this setting,  \refitm{Asymmetric} only prohibits symmetries that respect $Q$ (and similarly for \refitm{Simple}). It follows that if $M$ is asymmetric, then $(M,Q)$ will be asymmetric.

We can now give a recipe for constructing $N_n$ and $N_n^\mu$. Let $S$ be the closed surface of genus $2$. We start with a pair of simple, asymmetric caps for $S$, denoted $T$ and $B$ (for \emph{top} and \emph{bottom}), which are required to be distinct up to homeomorphism. In addition, choose a pseudo-Anosov homeomorphism $\varphi: S \to S$ and a pants decomposition $Q$. Then, let 
\begin{equation}\label{Eqn:GluingFormula}
N_n = B \cup_{\tau_n \circ (\varphi^{2n}) \circ \tau'_n} T
\qquad \text{and} \qquad
N_n^\mu = B \cup_{\tau_n \circ (\varphi^{2n}) \circ \tau'_n \circ \mu} T.
\end{equation}
Here, the subscript in \refeqn{GluingFormula} denotes the gluing map from $\bdy B$ to $\bdy T$. 
%    \marginpar{Check marking in above formula.}
Each of $\tau_n$ and $\tau'_n$ is a product of large Dehn twists about the curves of $Q$, where the notion of \emph{large} grows with $n$. Meanwhile $\mu$ is the (unique) \emph{hyper-elliptic involution} on $S$, which is central in $MCG(S)$. See \refsec{Mutation} for more on hyper-elliptic involutions.

To emphasize the flexibility of the construction, we assert that essentially \emph{any} sequence of manifolds constructed using the recipe \refeqn{GluingFormula} will satisfy the claims of \refthm{GeoSimManifolds}, up to a linear factor depending on $\varphi$. For $n \gg 0$, the middle portion of $N_n$ (between the curves being twisted) will become almost--isometric to the pared acylindrical manifold $W_{\varphi,n}$, which in turn converges to $\widetilde{M}_\varphi$ as $n \to \infty$. Thus all the results of Sections~\ref{Sec:Kleinian} and \ref{Sec:UniformCount} will apply to the geometry of $N_n$ and $N_n^\mu$.

In \refsec{Construct}, we give a very explicit construction of the caps $T$ and $B$, along with slightly restricted choices for the pants decomposition $Q$ and the gluing map $\varphi$. The point of the explicit choices is to make it easier to verify that $T$ and $B$ are indeed simple and asymmetric. In addition, we will drill out the curves of the pants decomposition in order to make the gluing of caps more rigid and to plug into the commensurability criterion of \refthm{PantedCriterion}.

%%%%%%%%%%%%%%%%%%%%%%%%%%%%%%%%%%%%%%%%%%%%%%%%%%%%%%%%%%%%%%%%%%%%%%%%%%%%%%%%%%%

\subsection{An explicit construction} \label{Sec:Construct}

Now, we give explicit choices for all the pieces used to build $N_{n}$. We begin with an explicit construction of $T$ and $B$.

Let $H$ be the handlebody of genus $2$, and let $K \subset H$ be the knot depicted in \reffig{SpaghettiMonster}. We let $T$ be the result of $5/3$ surgery on $K$, while $B$ is the result of $7/4$ surgery on $K$. 
For general $p/q$ surgery, we have the following result.

\begin{figure}
\begin{overpic}[width=3in]{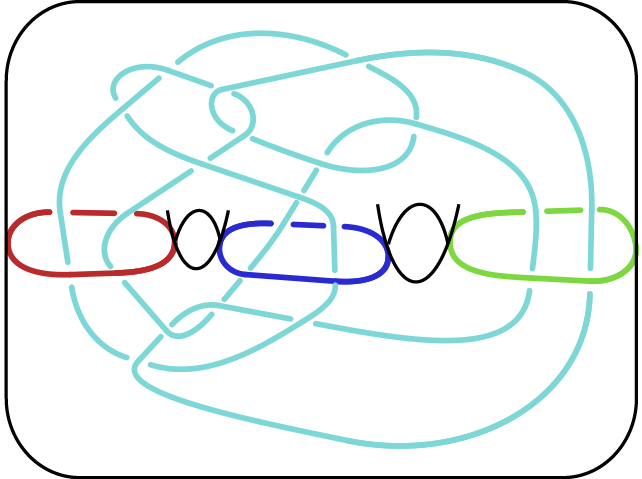}
\put(8,9){$H$}
\put(32,13){$K$}
\put(101,60){$P$}
\put(101,15){$P'$}
\end{overpic}
\caption{To construct the top cap $T$, start with the knot $K$ in a genus $2$ handlebody $H$. 
Then perform $5/3$ Dehn surgery on $K$ to get a simple, asymmetric $3$--manifold. Then, add a paring locus along three simple closed curves on $\bdy H$, as shown. The resulting pared manifold $(T, Q_T)$ has totally geodesic boundary along $P \cup P'$.}
\label{Fig:SpaghettiMonster}
\end{figure}
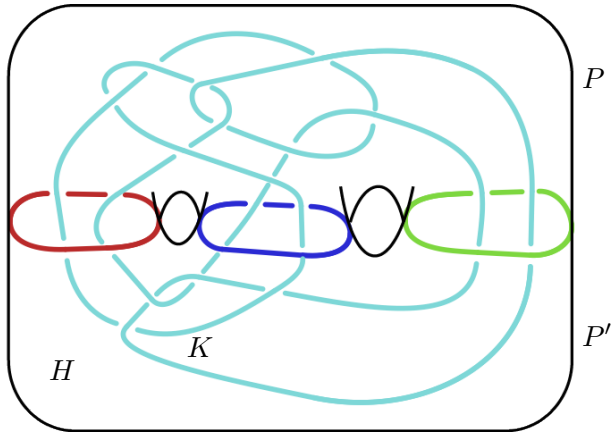

\begin{lemma}\label{Lem:SimpleCaps}
Let $H$ be the handlebody of genus $2$, let $S = \bdy H$, and let $K \subset H$ be the knot depicted in \reffig{SpaghettiMonster}.
Then, for every slope $p/q$ with $q > 2$, the $3$--manifold $H(p/q)$ obtained by $p/q$ surgery on $K$ is a simple cap for $S$. 

In particular, $T$ and $B$ are simple caps for $S$.
\end{lemma}

\begin{proof}
We begin by checking that $H \setminus K$ is a simple $3$--manifold. This can be accomplished by doubling $H \setminus K$ along $\bdy H = S$ to obtain the complement of a $2$--component link in $(S^2 \times S^1) \# (S^2 \times S^1)$. SnapPy verifies that this link complement is hyperbolic, and does not have any symmetries apart from the reflection along a totally geodesic copy of $\bdy H$. 
By the easy direction of Thurston's hyperbolization, it follows that $H \setminus K$ is simple.
We remark that this argument also shows $H \setminus K$ is asymmetric.

Next, observe that $1/0$ surgery along $K$ yields the handlebody $H$, which definitely contains essential disks. Then a combination of theorems by Gordon, Luecke, Scharlemann, and Wu (see \cite[Table 2.1]{gordon:small-surfaces}) says that any $p/q$ surgery with $q > 2$ will produce a $3$--manifold without any essential spheres, tori, disks, or annuli. Thus all these manifolds are simple.
\end{proof}

Next, we fix a pants decomposition $Q \subset \bdy H \cong S$ consisting of the 
the red, blue, and green curves in Figure \ref{Fig:SpaghettiMonster}. The construction of $T$ via $5/3$ surgery on $K \subset H$ identifies $\bdy T$ with $\bdy H$, allowing us to designate $Q$ to be the paring locus of $\bdy T$. Similarly, the construction of $B$ via $7/4$ surgery on $K \subset H$ allows us to designate $Q$ as the paring locus of $\bdy B$. To avoid confusion on issues of marking, we refer to the resulting pared manifolds as $(T, Q_T)$ and $(B, Q_B)$.
Since $T$ and $B$ are simple by \reflem{SimpleCaps}, this construction makes $(T, Q_T)$ and $(B,Q_B)$ into pared, acylindrical $3$--manifolds with totally geodesic boundary consisting of pairs of pants. Using SnapPy, we compute that
\begin{equation}\label{Eqn:TBVolume}
\vol(T,Q_T) = 36.4979 \ldots
\qquad \mbox{and} \qquad
\vol(B,Q_B) = 36.5377 \ldots \, .
\end{equation}

\begin{lemma}\label{Lem:AsymmetricCaps}
The pared manifolds $(T,Q_T)$ and $(B,Q_B)$ are simple, asymmetric caps for the pair $(S,Q)$. Furthermore, these caps are not isometric. 
\end{lemma}  

\begin{proof}
This follows from rigorous verification routines included in SnapPy \cite{snappy:rigor}. The program can rigorously verify the canonical cell decomposition of a hyperbolic $3$--manifold, which enables it to find all symmetries and certify that two $3$--manifolds are not isometric.

To assist SnapPy in this endeavor, it is convenient to isometrically embed $(T, Q_T)$ in a finite volume cusped $3$--manifold, as follows.
If the handlebody $H$ is embedded in $S^3$ as shown in \reffig{SpaghettiMonster}, the pair of pants $P \subset \bdy H$ becomes isotopic to $P' \subset \bdy H$ through $S^3 \setminus H$. Thus we may glue $P$ to $P'$, realizing $K \cup Q$ as the $4$--component link shown in \reffig{SpaghettiMonster}. Performing $(5,3)$ surgery on $K \subset S^3$ results in a $3$--cusped manifold, which becomes isometric to  $(T, Q_T)$ after cutting along a pair of pants to separate $P$ from $P'$.  An identical construction works for $(B, Q_B)$.
\end{proof}

Now, we perform the following sequence of steps to build the $3$--manifold $N_n$. We refer the reader to Figure \ref{Fig:BuildingNn} for a visual description of this process.

\begin{figure}
	\begin{overpic}[width=4in]{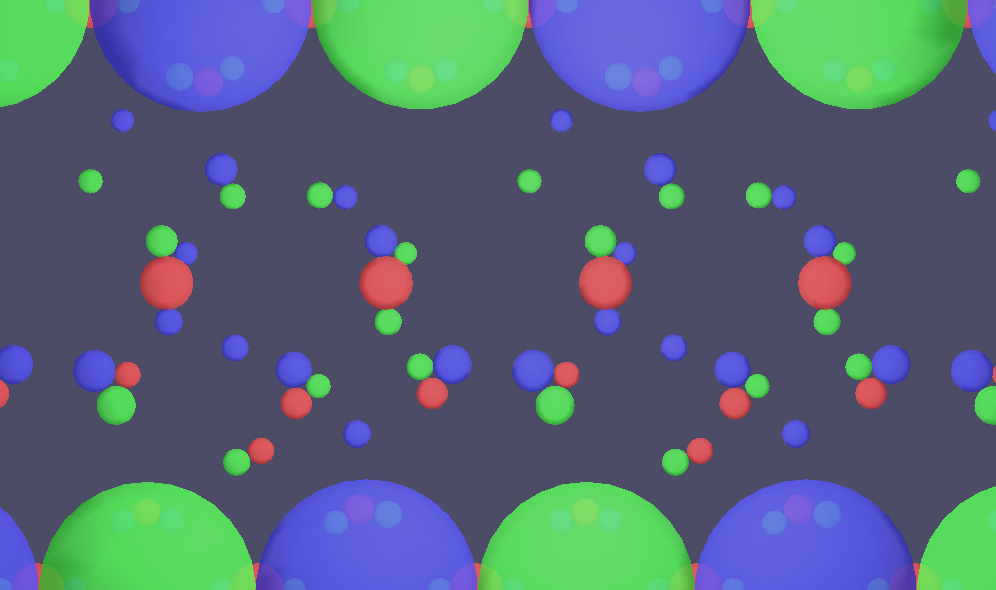}
	\end{overpic}
	\caption{A horoball diagram for the top cap $(T,Q_T)$. The view is from the red cusp in \reffig{SpaghettiMonster}. The distinguished line of pairwise tangent full-sized horoballs corresponding to $P$ is at the top, while that of $P'$ is at the bottom. 
Note that the full-sized horoballs along the top and bottom of the figure are not tangent to any other full-sized horoballs in the interior of the figure. In the language of \refdef{Isolated}, $P$ and $P'$ are \emph{geometrically isolated on one side}.
	The pattern of smaller horoballs shows that there are no symmetries interchanging $P$ with $P'$; this is also confirmed with a rigorous computation.
	}
	\label{Fig:CapCuspView}
\end{figure}

\begin{enumerate}[1.]
\item \underline{Construct $W_{\varphi, n}$:}
Let $Q \subset S$ be the pants decomposition specified above. Choose 
a pseudo-Anosov $\varphi:S \rightarrow S$. As in Sections~\ref{Sec:Pared} and \ref{Sec:DoubleLimit}, the triple $(Q,\varphi,n)$ specifies a pared acylindrical manifold $W_{\varphi, n} = M_{S}(\varphi^{-n}Q, \varphi^{n}Q)$. 
Recall that $M_{S}(\varphi^{-n}Q, \varphi^{n}Q)$ is equipped with a marking by $S$, which comes by including $S$ as $S \times \{ 1/2 \}$. In this marking, the bottom paring locus is $\varphi^{-n}Q$, and the top paring locus is $\varphi^{n}Q$.

As discussed in Remark \ref{Rem:CoarseConstants}, we choose a pseudo-Anosov $\varphi$  that satisfies $4 E t_{\Curve}(\varphi) > 1$. Here, $E$ is the constant of \refeqn{LowerElectricDist}, while $t_{\Curve}(\varphi)$ is the translation distance of $\varphi$ in the curve complex $\Curve(S)$, as in \refeqn{StableLength}. 
Our choice of $\varphi$ means we have $B_\varphi > 1/2$ in \refprop{LinearGeometryGrowth}.

\smallskip

\item \underline{Construct $J_{n}$:}
We build a finite-volume hyperbolic $3$--manifold $J_n$ as follows. Glue the lower boundary of $W_{\varphi, n}$ to the bottom cap $(B, Q_B)$. This joins the rank--1 cusps of $\varphi^{-n}(Q) \subset \bdy W_{\varphi, n}$ with the rank--1 cusps of $Q_B \subset \bdy B$, turning them into rank--2 cusps with a torus cross-section. In a similar manner, glue the upper boundary of  $W_{\varphi, n}$ to the top cap $(T, Q_T)$, joining rank--1 cusps to form rank--2 cusps. All of the gluing occurs by isometry along rigid, totally geodesic pairs of pants.

This construction endows each cusp of $J_n$ with a canonical \emph{surface--framed longitude}. This is the slope along which an annulus parallel to $Q_T$ is joined to an annulus parallel to $\varphi^n Q$ (and similarly for $Q_B$ and $\varphi^{-n} Q$). We also choose a \emph{meridian} to be a slope intersecting the longitude once.

\smallskip

\item
\underline{Construct $M_{n}$:} This $3$--manifold is obtained from $J_n$ by performing sufficiently long $1/k_i$ Dehn fillings along the three cusps that meet $\bdy B$. The surgery coefficient $1/k_i$, with respect to the meridian--longitude framing chosen above, ensures that the filling slope intersects the surface--framed longitude exactly once. Equivalently, $1/k_i$ Dehn filling means that we are gluing the un-pared cap $B$ to $S \times \{0 \}$ via $k_i$  Dehn twists along the $i$-th curve of $\varphi^{-n} Q$. The integers $k_i$ must be sufficiently large so that the curves of $\varphi^{-n} Q$ become the three shortest geodesics in $M_n$.

By construction, $M_n$ contains an isometrically embedded copy of $(T, Q_T)$, whose boundary in $M_n$ consists of totally geodesic pants $P$ and $P'$. We call this submanifold $M_+ \subset M_n$, and note that the geometry of $M_+ \isom  (T, Q_T)$ is independent of $n$. By contrast, the remainder of $M_n$ is a submanifold $M_{-}$ whose geometry changes with $n$.

\smallskip

\item 
\underline{Construct $N_{n}$:}
This $3$--manifold is obtained from $M_n$ by performing sufficiently long and sufficiently different $1/k'_i$ Dehn filling on either two or all three cusps of $M_n$. Here, the surgery coefficient has the same meaning as above: the filling slopes realize some number of Dehn twists along the curves of $\varphi^n Q$. The meaning of \emph{sufficiently long} and \emph{sufficiently different} comes from \refthm{PantedCriterion}, and more specifically from Equation~\refeqn{UnequalLengths}: we need the curves of $\varphi^n Q$ to be the three shortest geodesics in $N_n$, and their lengths to have large ratios.

We could have combined the last two steps, by Dehn filling $J_{n}$ to obtain $N_n$ directly.  However, building $M_n$ (with its submanifolds $M_+$ and $M_{-}$) will allow us to apply \refthm{PantedCriterion} and show that $N_n$ is minimal in its commensurability class.

Because of our choice of Dehn filling coefficients  (which realize Dehn twists along $S$), $N_n$ decomposes as the union of caps homeomorphic to $B$ and $T$, connected via a product region homeomorphic to $S \times I$. 
This product region is glued to $B$ along all of $S \times \{ 0 \}$, and to $T$ along either all of $S \times \{1 \}$ or the complement of one closed curve in this surface. See \reffig{BuildingNn}. We will continue to think of $B$ and $T$ as submanifolds of $N_n$.
Note as well that this construction of $N_n$ satisfies the recipe of \refeqn{GluingFormula}.

\end{enumerate}

\begin{figure}
\begin{overpic}[scale=0.50]{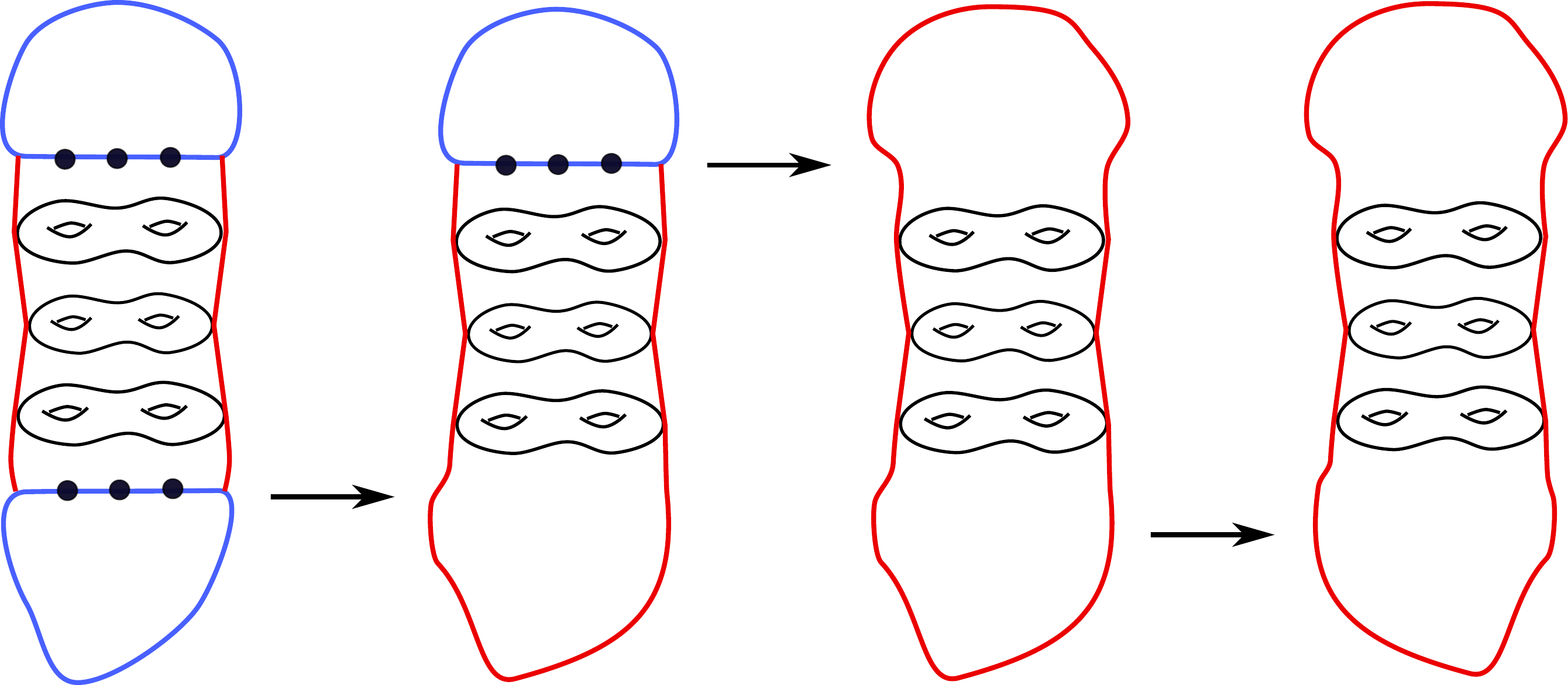}
\put(7,37){{$T$}}
\put(6,6){{$B$}}
\put(16,25){{$W_{\varphi, n}$}}
\put(11, 0){\Large{$J_{n}$}}
\put(16,14){Dehn fill}
\put(34,8){{$M_{-}$}}
\put(34, 37){\large{$M_{+}$}}
\put(44,35){Dehn fill}
\put(40, 0){\Large{$M_{n}$}}
\put(69, 0){\Large{$N_{n}$}}
\put(73,11){Mutate}
\put(71,22){$F_n$}
\put(99,22){$F_n$}
\put(96,0){\Large{$N_n^\mu$}}
\end{overpic}
\caption{A schematic summary of the construction of $N_{n}$ and $N_n^\mu$. Start with $J_{n}$ (left), which is built by gluing caps $T$ and $B$ to the ends of the product region $W_{\varphi, n}$. Next, Dehn fill $J_n$ along the bottom $3$ cusps to obtain $M_{n}$. Then, Dehn fill along either two or all three of the top cusps to obtain $N_{n}$. Finally, we can obtain $N_n^\mu$ by mutating along a genus two surface $F_n$. In our diagram, the bottom caps of $N_n$ and $N_n^\mu$ differ by a reflection.}
\label{Fig:BuildingNn}
\end{figure}
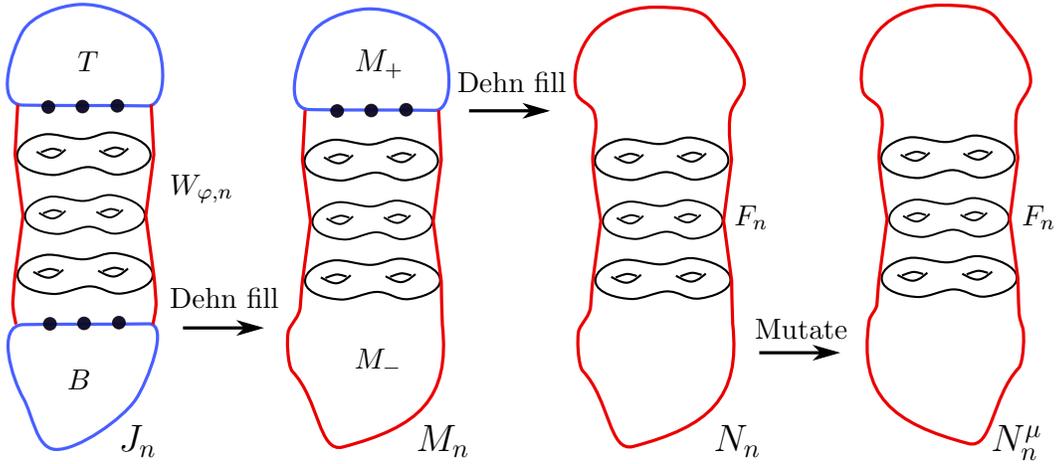

\noindent Given $J_n$ and $N_n$, we construct two auxiliary objects that will be useful in our arguments.

\begin{enumerate}[1.]
\setcounter{enumi}{4}

\item 
\underline{Construct $\Sigma_\pm \subset N_n$:} Recall that $\bdy (B, Q_B)$ consists of two pairs of pants, which are totally geodesic in $J_n$. Choose an orientation on the three curves of $Q_B$ and an ideal triangulation of the complementary pairs of pants. Each endpoint of an ideal edge can be spun about a curve of $Q_B$, following the prescribed orientation of that curve. This gives a finite, filling lamination on $S$, which we call $\lambda_-$.

After $J_n$ is filled to obtain $N_n$, Thurston's theorem \cite[Proposition 9.7.1]{thurston:notes} says that we may pleat $S$ along $\lambda_-$ to obtain a pleated surface $\Sigma_-$. (Recall \refdef{PleatedSurf}.) For long Dehn fillings, the geometry of $\Sigma_-$ closely approximates the geometry of $\bdy (B, Q_B) \subset J_n$.

We perform the same construction on $\bdy (T, Q_T)$ to get a finite, filling lamination $\lambda_+$. Once $J_n$ is filled to obtain $N_n$, this gives a pleated surface $\Sigma_+ \subset N_n$. (If all three curves of $Q_T$ are filled in $N_n$, the pleated surface $\Sigma_+$ will be homotopic to $S$; otherwise, it will be the complement of a curve in $S$.)

\smallskip

\item 
\underline{Construct $X_{n}$:}
Recall that $W_{\varphi,n} \in AH(S)$ is marked by an embedded copy of the surface $S$. This copy of $S$ survives as an essential surface in $J_n$ and then in $N_n$. (This is because our choice of $1/k_i$ Dehn filling is realized by Dehn twists along $S$. A compression disk for $S$ in $N_n$ would have to contain a compression disk in either $B$  or $T$. But $B$ and $T$ have no compression disks, by \reflem{SimpleCaps}.) Furthermore, the topology of $B$ and $T$ prevents $S$ from being a fiber or virtual fiber.

Let $X_n$ be the cover of $M_n$ corresponding to $S$. This manifold is marked by $S$, placing $X_n \in AH(S)$. By the work of Thurston and Bonahon \cite{Bonahon:ends}, $X_n$ is geometrically finite. Its convex core contains an isometric copy of $\Sigma_+$ close to $\bdy_+ \core X_n$, and an isometric copy of $\Sigma_-$ close to $\bdy_- \core X_n$.

\end{enumerate}

\begin{proposition} \label{Prop:NnGeometry}
For each $n$, there is a length cutoff $z_n$, such that if all slopes in the Dehn filling $J_n \to N_n$ are longer than $z_n$, the following will hold.
\begin{enumerate}
\item\label{Itm:VolGrowth} There are positive constants $A_\varphi$ and $A'_\varphi$, depending on $\varphi$, such that for all $n \gg 0$,
\begin{equation*}%\label{Eqn:NnVolume}
A_\varphi \leq \frac{\vol(N_n)}{n} \leq A'_\varphi.
\end{equation*}

\item\label{Itm:ShortestCurves} The cores of the Dehn filling solid tori are the shortest closed geodesics in $N_n$.

\item\label{Itm:NnPleated} The pleated surfaces $\Sigma_\pm \subset N_n$  satisfy $d(\Sigma_-, \Sigma_+) > n/2$.

\item\label{Itm:XnConverge} For an appropriate choice of baseframes, $X_n$ converges strongly to $\widetilde{M}_\varphi$.
\end{enumerate}
\end{proposition}

\begin{proof}
For \refitm{VolGrowth}, recall that $J_n$ was constructed by gluing three pieces by isometry along their boundaries. Thus, by \refeqn{TBVolume},
\[
\vol (J_n) = \vol(W_{\varphi,n}) + \vol(B,Q_B) + \vol(T,Q_T) =  \vol(W_{\varphi,n}) + 73.035 \ldots \, .
\]

Let $s_1^n, \ldots, s_6^n$ be the Dehn filling slopes on the cusps of $J_n$ that produce $N_n$. Each $s_i^n$ has a normalized length $\hat{L}_i^n$, as in \refdef{NormalizedLength}.
 Neumann and Zagier showed that
\begin{align*}
 \vol(N_n) & = \vol(J_n) - \pi^2 \left( \sum \big( \hat{L}_i^n \big)^{-2} \right) + O \left( \sum \big( \hat{L}_i^n \big)^{-4} \right) \\
 & = \vol(W_{\varphi,n}) + O(1) - \pi^2 \left( \sum \big( \hat{L}_i^n \big)^{-2} \right) + O \left( \sum \big( \hat{L}_i^n \big)^{-4} \right).
\end{align*}
See \cite[Theorem 1A]{NeumannZagier}. Thus, if every normalized length $\hat{L}_i^n$ is sufficiently large, the difference $| \!\vol(N_n) - \vol(W_{\varphi,n}) |$ will be uniformly bounded. Since $\vol(W_{\varphi,n})/n$ is bounded above and below in Equation~\refeqn{LinearVolume} of \refprop{LinearGeometryGrowth}, it follows that $\vol(N_n)/n$ satisfies the same bounds with slightly modified values of the constants $A_\varphi$ and $A'_\varphi$. This proves \refitm{VolGrowth}.

\smallskip

Conclusion \refitm{ShortestCurves} also follows from the work of Neumann and Zagier \cite[Proposition 4.3]{NeumannZagier}. This is because each normalized length $\hat{L}_i^n$ predicts the length of the corresponding closed curve in $N_n$, via Equation~\refeqn{NeumannZagier}. Thus choosing $\hat{L}_i^n$ sufficiently long forces the cores of the surgery solid tori to be very short.

\smallskip

Conclusion \refitm{NnPleated} follows from the $K$--bilipschitz theorem of Brock and Bromberg \cite[Theorem 1.3]{brock-bromberg:density}; see also Magid \cite[Theorem 1.2]{Magid:deformation}. They show that, for every bilipschitz constant $K > 1$, there is a length cutoff $z = z(K)$ depending only on $K$, such that if $\hat{L}_i > z(K)$ for all $i$, there will be a $K$--bilipschitz embedding  $\psi_n: \core^0 W_{\varphi, n} \hookrightarrow N_n$. Furthermore, by \cite[Theorem 6.10]{brock-bromberg:density}, a geodesic on $\bdy W_{\varphi, n}$ is mapped by $\psi_n$ to a curve whose geodesic curvature is vanishingly small. This is a strong, quantified version of geometric convergence.

Now, recall the finite, geodesic laminations $\lambda_\pm \subset \bdy W_{\varphi, n}$. By the above paragraph, their images $\psi_n(\lambda_\pm)$ are nearly geodesic in $N_n$. Thus $\psi_n$ maps the geodesic pairs of pants comprising $\bdy_+ W_{\varphi, n}$ vanishingly  close to $\Sigma_+$, and similarly for $\Sigma_-$.

In constructing $W_{\varphi,n}$, we chose a mapping class $\varphi$ to have large translation distance in the curve complex $\Curve(S)$, ensuring that $B_\varphi > 1/2$ in \refprop{LinearGeometryGrowth}. Thus, by  \refprop{LinearGeometryGrowth},  for $n \gg 0$ we have
\[ d(\bdy_- \core^0 W_{\varphi, n}, \, \bdy_+ \core^0 W_{\varphi, n}) > n/2.\]
(In this section, $d_0(\cdot, \cdot)$ is the same as ordinary distance because $S$ is a closed surface.) By choosing a bilipschitz constant $K = K_n$ sufficiently close to $1$, and taking $z_n = z(K_n)$, we ensure that this lower bound is preserved in $N_n$, hence $d(\Sigma_-, \Sigma_+) > n/2$.

\smallskip

For \refitm{XnConverge}, recall from \refprop{DoubleLimit} that there is a choice of basepoints $x_n \in W_{\varphi,n}$, with baseframes $\omega_n$ at $x_n$, such that we have a geometric limit $(W_{\varphi,n}, \omega_n) \to (\widetilde{M}_\varphi, \omega_\infty)$. The distance from $x_n$ to the totally geodesic boundary of $ W_{\varphi,n}$ must go to $\infty$ with $n$, for otherwise, this boundary would appear in the geometric limit. So long as $d(x_n, \bdy W_{\varphi,n}) > R$, an $R$--ball about $x_n \in J_n$ is contained entirely in $W_{\varphi,n}$. 
Thus, by 
the Gromov--Hausdorff characterization of geometric limits in \refeqn{GeometricLimit}, we have a geometric limit $(J_n, \omega_n) \to (\widetilde{M}_\varphi, \omega_\infty)$. 
Similarly, for very long Dehn fillings, an $R$--ball about $x_n \in J_n$ is nearly isometric to a metric $R$--ball about $\psi_n(x_n) \in N_n$, where $\psi_n: \core^0 W_{\varphi, n} \hookrightarrow N_n$ is the bilipschitz diffeomorphism described above. This will enable us to conclude that $N_n$ converges geometrically to $\widetilde{M}_\varphi$. Furthermore, since all of the topology of this $R$--ball is captured by $\pi_1(X_n)$, we will in fact have a geometric and algebraic limit $X_n \to  \widetilde{M}_\varphi$.

To make this summary precise, recall that the Chabauty topology is given by a metric $d_{\text{Chab}}$ \cite{Biringer:Chabauty}.
%    \cite[Proposition 3.1.2]{CaEpGr}. 
Thus we may write 
\[
d_{\text{Chab}} \left( (J_n, \omega_n) , \, (\widetilde{M}_\varphi, \omega_\infty) \right) = \epsilon_n \to 0.
\]
By Thurston's Dehn surgery theorem (see e.g. \cite[Theorem E.5.1]{BP}), any sequence of Dehn fillings of $J_n$ in which the lengths of the filling slopes approach $\infty$ will converge to $J_n$ in the Chabauty topology. Let $\omega_n \in W_{\varphi,n} \subset J_n$ be the same set of baseframes as before. Then there is a length cutoff $z_n$ such that if $\widehat{L}_i^n \geq z_n$ for all $i$, we will have baseframes $\nu_n$ on $N_n$ such that
\[
d_{\text{Chab}} \big( (J_n, \omega_n) , \, (N_n, \nu_n) \big) \leq \epsilon_n.
\]
Combining the last two equations gives
\[
d_{\text{Chab}} \left(  (N_n, \nu_n), \,  (\widetilde{M}_\varphi, \omega_\infty) \right) \leq 2 \epsilon_n \to 0,
\]
hence there is a geometric limit $(N_n, \nu_n) \to (\widetilde{M}_\varphi, \omega_\infty) $, as claimed.

For every $R > 0$, we know that there is an integer $n(R)$ such that $d(x_n, \bdy W_{\varphi,n}) > 2R$ for all $n \geq n(R)$. By the bilipschitz theorem of Brock and Bromberg \cite[Theorem 1.3]{brock-bromberg:density}, and the above argument, we have $d(\psi_n(x_n), \Sigma_\pm) > R$. Since the entire region between $\Sigma_-$ and $\Sigma_+$ has an isometric lift to $\core (X_n)$, we know that the metric $R$--ball about $\psi_n(x_n)$ also lifts isometrically to $\core (X_n)$ for $n \geq n(R)$. Since this works for every $R > 0$, we have a geometric limit $(X_n, \nu_n) \to (\widetilde{M}_\varphi, \omega_\infty) $.

Finally, observe that each of $X_n$, $W_{\varphi,n}$, and $\widetilde{M}_\varphi$ is marked by an inclusion of $S$.  For long Dehn fillings, the representations $\pi_1(S) \to \pi_1(X_n) \subset \pi_1(N_n)$ converge algebraically in $AH(S)$  to the limiting representation given by the marking of $W_{\varphi,n}$. Since $W_{\varphi,n}$ converges strongly to $\widetilde{M}_\varphi$, we have a strong limit $(X_n, \nu_n) \to (\widetilde{M}_\varphi, \omega_\infty) $.
\end{proof}

%%%%%%%%%%%%%%%%%%%%%%%%%%%%%%%%%%%%%%%%%%%%%%%%%%%%%%%%%%%%%%%%%%%%%%%%%%%%%%%%%%%%%%%%%%%%%%%%%%%%%

\subsection{Mutant partners}
\label{Sec:Mutation}

We now describe a procedure, called \emph{mutation}, that will preserve much of the geometry of $N_n$ while changing its commensurability class.

Given a surface $S$, a \emph{hyper-elliptic involution} is a self-homeomorphism $\mu: S \to S$ that preserves the isotopy class of every simple closed curve. A hyperbolizable surface $S$ admits hyper-elliptic involutions if and only if $\chi(S) = -1$ or $-2$. Since a marked hyperbolic structure on $S$ is determined by the lengths of finitely many closed curves, and $\mu$ preserves all of these lengths, it acts trivially on $\Teich(S)$, the compactification $\overline{\Teich(S)}$, and the product $\overline{\Teich(S)} \times \overline{\Teich(S)}$.
Consequently, hyper-elliptic involutions of $S$ act trivially on $AH(S)$.

\begin{definition}\label{Def:Mutation}
Let $M$ be a $3$--manifold and $F \subset M$ an incompressible, boundary--incompressible surface admitting a hyper-elliptic involution $\mu$. We may cut $M$ along $F$ and reglue via $\mu$. This creates a $3$--manifold $M^{\mu}$, called the \emph{mutant partner} of $M$. The process is called \emph{mutation along $F$}. 
% We will often refer to the two manifolds $M$ and $M^{\mu}$ as \emph{mutant partners}. 
\end{definition}

We can now describe how to create mutant partners for our manifolds. First, let $F_{n} \subset W_{\varphi, n} \subset J_n$ denote the copy of $S$ marking $W_{\varphi, n}$, i.e., $F_{n} = S \times \{\frac{1}{2}\} \subset  S \times I$. Since $S$ is the closed surface of genus $2$, there is a unique hyper-elliptic element $\mu \in MCG(S)$.
We obtain $J_{n}^{\mu}$ from $J_{n}$ by mutating along $F_{n}$. Note that the hyper-elliptic involution $\mu$ extends over $ W_{\varphi, n} \in AH(S)$. Thus $J_n^\mu$ is built from the same pieces as $J_n$, with a modified gluing map. To build $J_n^\mu$, we glue $W_{\varphi, n}$ to $(T,Q_T)$ by the same map as before, and glue it to $(B,Q_B)$ by the previous map composed with $\mu$.

By definition, the hyper-elliptic involution $\mu$ preserves the curves of $\varphi^{\pm n} Q$, which correspond to the cusps of $J_n$. We construct $M_{n}^{\mu}$ and $N_{n}^{\mu}$ by Dehn filling $J_n^\mu$ along exactly the same filling slopes used to produce $M_n$ and then $N_n$.
After the Dehn filling, $N_n$ still contains an embedded, incompressible copy of $F_n$. (See \reffig{BuildingNn}.) Thus we could also obtain $N_n^\mu$ by first filling $J_n$ to obtain $N_n$, and then mutating along $F_n$ to obtain $N_n^\mu$. Because the filling coefficients are unchanged, the operations of Dehn filling and mutation commute. 

The covers of $N_n$ and $N_n^\mu$ corresponding to $F_n$ are both isometric to $X_n$. In fact, the hyper-elliptic involution $\mu:S \to S$ extends to $X_n \cong S \times \RR \in AH(S)$. This can be rephrased to say that although $N_n$ and $N_n^\mu$ are incommensurable (as we will see in \reflem{NnMinimal}), they do share a common infinite--sheeted cover, which is quasifuchsian if $N_n$ is closed.

Before checking the properties claimed in \refthm{GeoSimManifolds}, we establish the following fact.

\begin{lemma}\label{Lem:OnlyPants}
When $n \gg 0$, the only essential pairs of pants in $J_n$ are isotopic into $\bdy W_{\varphi,n}$. The only essential pairs of pants in $M_n$ are isotopic into $\bdy M_+ = \bdy M_-$. Furthermore, the same statements hold in $J_n^\mu$ and $M_n^\mu$.
\end{lemma}

\begin{proof}
First, we claim that the only essential pairs of pants in  $(T,Q_T)$ or $(B,Q_B)$ are isotopic to the boundary. This is verified using Regina \cite{burton:regina}, by enumerating all orientable normal surfaces with Euler characteristic $-1$ in these manifolds. 

Next, consider an essential pair of pants $P \subset W_{\varphi,n}$, with boundary mapped to the parabolic locus. 
If boundary components of $P$ are mapped to both $\varphi^{-n} Q$ and $\varphi^n Q$,  the $\epsilon$--thick part of $P$ (where $\epsilon$ is the Margulis constant) will necessarily contain a path from $\bdy_- \core^0 W_{\varphi,n}$ to $\bdy_+ \core^0 W_{\varphi,n}$. But the $\epsilon$--thick part of $P$ has bounded diameter (compare \reflem{ClosestHoroball}), whereas the distance between the upper and lower boundary of $\core^0 W_{\varphi,n}$ must grow linearly with $n$ by \refprop{LinearGeometryGrowth}. For $n \gg 0$, this is a contradiction. Thus, for large $n$, 
all three boundary components of $P$ are mapped to curves of $\varphi^n Q$, or all three to $\varphi^{-n} Q$.
In either case, $P$ is isotopic to one of the pairs of pants comprising $\bdy W_{\varphi,n}$. 

Now, consider an essential pair of pants in $P$ in $J_n$ or $J_n^\mu$. The intersection $P \cap \bdy W_{\varphi,n}$ must consist of simple closed curves that are essential in $P$. But the only essential simple curves in a pair of pants are peripheral, hence $P$ can be isotoped to be disjoint from $\bdy W_{\varphi,n}$. After this isotopy, $P$ is entirely contained in one of the pieces $T$, $B$, or $W_{\varphi,n}$. But we have already checked that any pair of pants in these submanifolds is isotopic into $\bdy W_{\varphi,n}$. 

Finally, consider an essential pair of pants in $P \subset M_n$. By the same intersection argument as above, $P$ can be isotoped to be disjoint from $\bdy M_+ = \bdy M_-$. If $P \subset M_+ = (T, Q_T)$, then we have already checked that $P$ is isotopic into $\bdy M_+$. If $P \subset M_-$, observe again that $P \cap \core^0 M_n$ has bounded diameter. Thus all of $P$ is contained in a collar neighborhood of $\bdy M_-$, homeomorphic to $S \times I$. Now $P$ must be isotopic into $\bdy M_-$, by the same argument as for $W_{\varphi,n}$.
\end{proof}

\begin{lemma}\label{Lem:MutantVolume}
For $n \gg 0$, the mutant partners $N_{n}$ and $N_{n}^{\mu}$ are non-isometric hyperbolic $3$-manifolds. Furthermore, $\vol(N_n) = \vol(N_n^\mu)$.
\end{lemma}

\begin{proof}
Suppose, for a contradiction, that there is an isometry $\psi:N_n \to N_n^\mu$. By \refprop{NnGeometry}, the shortest closed geodesics in $N_n$ are the cores of the (five or six) solid tori that were added when we filled $J_n$. The isometry $\psi$ must respect these shortest geodesics, hence $\psi$ restricts to an isometry $\psi:J_n \to J_n^\mu$. 

Let $P \subset \bdy W_{\varphi,n}$ be a pair of pants. Then $\psi(P)$ is a totally geodesic pair of pants in $J_n^\mu$. By \reflem{OnlyPants}, all pairs of pants in $J_n^\mu$ occur along $\bdy W_{\varphi,n}$. 
% The same is true in $J_n^\mu$, since we have assumed it is isometric to $J_n$. 
Thus $\psi$ must respect the decomposition of $J_n^\mu$ into top and bottom caps, joined along the product region $W_{\varphi,n}$.  In particular, $\psi (T) \subset J_n^\mu$ must be a cap isometric to $(T,Q_T)$. Since the top and bottom caps are not isometric by \reflem{AsymmetricCaps}, it follows that $\psi(T)$ is the top cap of $J_n^\mu$, which we may identify with $(T,Q_T)$ in a unique way.

By \reflem{AsymmetricCaps}, the isometry $\psi \vert_T$ must be the identity map. By restricting to the boundary, $\psi \vert_{\bdy_+ W_{\varphi,n}}$ is also the identity map. Since we have performed a mutation along $F_n \subset W_{\varphi,n}$, it follows that $\psi \vert_{\bdy_- W_{\varphi,n}}$ is the hyper-elliptic involution $\mu$. But, by \reflem{AsymmetricCaps}, $\mu$ cannot extend over the bottom cap $(B, Q_B)$, giving a contradiction, Thus $N_n$ and $N_n^\mu$ are not isometric.

Finally, $\vol(N_n) = \vol(N_n^\mu)$ by a theorem of Ruberman \cite{Ru}.
\end{proof}

We remark that the argument of \reflem{MutantVolume} also shows that $J_n$ and $N_n$ have no symmetries. 
For, any isometry $\psi:J_n \to J_n$ must respect the decomposition of $J_n$ into top and bottom caps, joined along $W_{\varphi,n}$. But then the restriction $\psi \vert_T$ must be the identity map, hence $\psi$ is the identity.
It follows that $N_n$ cannot be a regular cover of any orbifold (other than itself). In fact, we can show more.

\begin{lemma}\label{Lem:NnMinimal}
	For  $n \gg 0$, each of  $N_{n}$ and $N_{n}^{\mu}$ is non-arithmetic and minimal in its commensurability class. In  particular, $N_{n}$ and $N_{n}^{\mu}$ are incommensurable. 
\end{lemma}

\begin{proof}
This will follow as a consequence of \refthm{PantedCriterion}. We need to check the hypotheses of that theorem.

By construction, $M_n$ decomposes into submanifolds $M_+$ and $M_-$, where $M_+ \cong (T, Q_T)$. These submanifolds are glued along $P \cup P'$, which are the only pairs of pants in $M_n$ by \reflem{OnlyPants}.  The volume of $M_+$ is independent of $n$  (it is given in \refeqn{TBVolume}), whereas $\vol(M_n) > \vol(N_n)$ grows linearly with $n$ by \refprop{NnGeometry}. Thus $n \gg 0$ implies $\vol(M_-) \gg \vol(M_+)$. Furthermore, $M_+$ is asymmetric by \reflem{AsymmetricCaps}.

Next, we check that there is a choice of cusp neighborhoods $\{C_1, C_2, C_3 \}  \subset M_n$ such that the pairs of pants $P$ and $P'$ are pairwise tangent and geometrically isolated on one side (see \refdef{Isolated}). We choose a transverse orientation on both $P$ and $P'$, pointing away from $M_-$ and toward $M_+$. Then, we verify using SnapPy that in the horoball diagram of $M_+$, the full-sized horoballs of $(P \cup P') \cap C_i$ are not tangent to any other full-sized horoballs in the direction of $M_+$. See \reffig{CapCuspView}, and recall that 
the isometry class of $M_+ \cong (T, Q_T)$ does not depend on $n$.

Thus, by \refthm{PantedCriterion}, our choice of sufficiently long and sufficiently different Dehn filling slopes in the definition of $N_n$ ensures that $N_n$ is the unique minimal orbifold in its (non-arithmetic) commensurability class. 
The same argument applies to $M_n^\mu$ and $N_n^\mu$, because $M_n^\mu$ can be assembled from the same pieces $M_+$ and $M_-$ by modifying the gluing map by $\mu$. Thus $N_n^\mu$ is also non-arithmetic and minimal. By \reflem{MutantVolume}, $N_n$ and $N_n^\mu$ are not isometric, hence they are incommensurable.
\end{proof}

%%%%%%%%%%%%%%%%%%%%%%%%%%%%%%%%%%%%%%%%%%%%%%%%%%%%%%%%%%%%%%%%%%%%%%%

\subsection{Geodesics under mutation}\label{Sec:MutantGeodesics}

To complete the proof of \refthm{GeoSimManifolds}, it remains to check that $N_n$ and $N_n^\mu$ share the same set of closed geodesics up to length $n$, and that there are at least $e^n/n$ such geodesics. We verify this in the following two lemmas.

\begin{lemma}\label{Lem:SameGeodesicsMutation}
For $n \gg 0$, any closed geodesic $\gamma \subset N_n$ whose length is at most $n$ can be homotoped to be disjoint from the mutation surface $F_n$. Consequently, there is a bijection between the complex length spectra of $N_n$ and $N_n^\mu$ up to length $n$.
\end{lemma}

\begin{proof}
Recall the pleated surfaces $\Sigma_\pm \subset N_n$. Since $d(\Sigma_-, \Sigma_+) > n/2$ by \refprop{NnGeometry}, any closed geodesic $\gamma \subset N_n$ of length at most $n$ must be disjoint from one of these surfaces. If $\gamma$ is disjoint from $\Sigma_-$, we recall that $\Sigma_-$ is homotopic to $F_n$, hence $\gamma$ can be homotoped to be disjoint from $F_n$. If $\gamma$ is disjoint from $\Sigma_+$, then it is homotopic into $B$, which is disjoint from $F_n$. (The asymmetry between $\Sigma_-$ and $\Sigma_+$ only arises if we leave one cusp of $J_n$ unfilled, and $\Sigma_+$ meets that cusp.)

Since all geodesics of length at most $n$ can be homotoped to be disjoint from $F_n$, the conclusion about length spectra follows by a theorem of Millichap \cite[Proposition 4.4]{Mi2}.
\end{proof}

\begin{lemma}\label{Lem:GeodesicCountMutation}
For $n \gg 0$, each of $N_n$ and $N_n^\mu$ contains at least $e^n/n$ geodesics up to length $n$. That is,
\[
\pi_{N_n}(n) = \pi_{N_n^\mu}(n) \geq \frac{e^n}{n}.
\]
\end{lemma}

\begin{proof}
We begin with the hyperbolic manifold $X_n$, which is a covering space of both $N_n$ and $N_n^\mu$ corresponding to $F_n$. By \refprop{NnGeometry}, there is a choice of baseframes for which $X_n$ converges strongly to $\widetilde{M}_\varphi$. Applying \refprop{XnGeodesics} with $L = n$, we see that for $n \gg 0$,
\begin{equation}\label{Eqn:XnRestate}
\pi_{X_n}(n) \geq \frac{e^n}{n}.
\end{equation}

Every closed geodesic $\gamma \subset X_n$ projects to a closed geodesic in $N_n$. 
% The covering map $\psi: X_n \to N_n$ maps closed geodesics to closed geodesics. 
To complete the proof, we need to check that distinct geodesics in $X_n$ -- that is, distinct free homotopy classes of loops in $X_n$ -- project to distinct free homotopy classes in $N_n$. 

Suppose, for a contradiction, that $\gamma$ and $\gamma'$ are geodesic loops  in $X_n$ that are not freely homotopic in $X_n$, but become freely homotopic in $N_n$. We may suppose after a homotopy that $\gamma$ and $\gamma'$ lie in the marking surface $S$, which is embedded in $N_n$. In $N_n$, the free homotopy from $\gamma$ to $\gamma'$ is realized by a (singular) annulus $A$, which must have an essential intersection with $T$ or $B$. By Jaco's annulus theorem \cite[Theorem VIII.13]{jaco}, $T$ or $B$ must contain an essential embedded annulus $A'$, which contradicts \reflem{SimpleCaps}. Thus no annulus can exist, hence $\gamma$ and $\gamma'$ represent distinct geodesics in $N_n$.

By the same argument, distinct geodesics in $X_n$ project to distinct geodesics in $N_n^\mu$. In group--theoretic terms, we have checked that $\pi_1(X_n)$ is malnormal in both $\pi_1(N_n)$ and $\pi_1(N_n^\mu)$.
Combining this result with \refeqn{XnRestate} and \reflem{SameGeodesicsMutation} gives
\[
\pi_{N_n}(n) = \pi_{N_n^\mu}(n) \geq \pi_{X_n}(n) \geq \frac{e^n}{n}. \qedhere
\]
\end{proof}

\begin{proof}[Proof of \refthm{GeoSimManifolds}]
In the last several lemmas, we have checked that $N_n$ and $N_n^\mu$ satisfy all the criteria required by \refthm{GeoSimManifolds}. To recap: 
Conclusion \refitm{Volume} of the theorem holds by \refprop{NnGeometry} and \reflem{MutantVolume}. Conclusion \refitm{SameLengths} holds by \reflem{SameGeodesicsMutation}. Conclusion \refitm{ManyLengths} holds by \reflem{GeodesicCountMutation}. Finally, conclusion \refitm{Incommensurable} holds by \reflem{NnMinimal}.
\end{proof}

%%%%%%%%%%%%%%%%%%%%%%%%%%%%%%%%%%%%%%%%%%%%%%%%%%

\section{Spectrally Similar Knots}
\label{Sec:ArbKnots}

In this section, we will describe how to construct pairs of spectrally similar, incommensurable, hyperbolic knot complements that differ by mutation. In particular, we will outline the steps required to prove the following theorem.

\begin{named}	{\refthm{GeoSimKnots}}
For each $n \gg 0$, there exists a pair of non-isometric mutant hyperbolic knot exteriors $E_{n} = S^{3} \setminus K_{n}$ and  $E_{n}^{\mu} =  S^{3} \setminus K_{n}^{\mu}$ such that:
	
	\begin{enumerate}
\item $\vol(E_{n}) = \vol(E_{n}^{\mu})$, where this volume grows coarsely linearly with $n$. 
\item The (complex) length spectra of $E_{n}$ and $E_{n}^{\mu}$ agree up to  length at least  $2\log(n)$.
\item $E_{n}$ and $E_{n}^{\mu}$ have at least $n^2/(2 \log(n))$ closed geodesics up to length $2 \log(n)$.
\item $E_{n}$ is the unique minimal orbifold in its commensurability class, and the only knot complement in its commensurability class.
The same statement holds for $E_{n}^{\mu}$.
	\end{enumerate}
\end{named}

The construction behind \refthm{GeoSimKnots} is extremely similar in spirit to the construction in \refsec{GeoSimManifolds} that proves \refthm{GeoSimManifolds}. We will still use the recipe \refeqn{GluingFormula}, except this time the surface $S$ will be a $4$--holed sphere. The knot complement $E_n = S^3 \setminus K_n$ will have two caps, denoted $T$ and $B$, connected by a product region $S \times I$ whose diameter grows linearly with $n$. In this setting, the caps will be tangles with boundary $S$, and the product region corresponds to a large power $\varphi^{2n}$ of a pseudo-Anosov braid $\varphi$. Just as in \refsec{GeoSimManifolds}, $E_n^\mu = S^3 \setminus K_n^\mu$ will differ from $E_n$ by mutation. See \reffig{ArbKnots} for a visual summary.

\begin{figure}
	\begin{overpic}[scale=0.45]{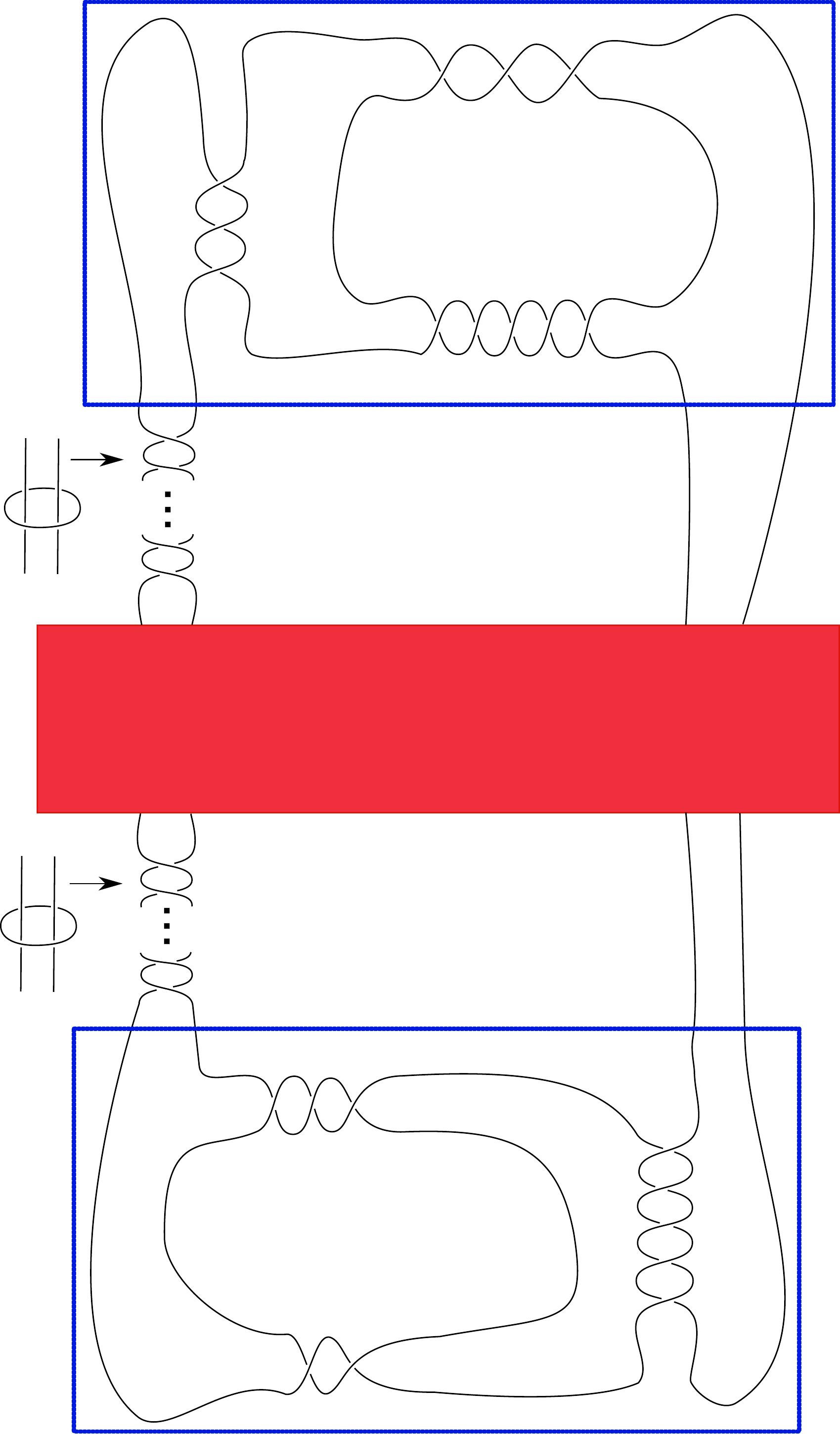}
		\put(29,49){\Large ${\varphi}^{2n}$}
		\put(1,85){\large $\bdy T$}
		\put(34,85){\large $T$}
		\put(0,13){\large $\bdy B$}
		\put(24,13){\large $B$}
%    		\put(19, 84){$t_{1}$}
%    		\put(35, 81){$t_{2}$}
%    		\put(37,90){$t_{3}$}
%    		\put(41,12){$t_{4}$}
%    		\put(21,18){$t_{5}$}
%    		\put(21,8){$t_{6}$}
		\put(15,64){$2x_n$}
		\put(-4,64){$Q_T$}
		\put(15,35){$2y_n$}
		\put(-4.5,35){$Q_B$}
	\end{overpic}
	\caption{A schematic diagram of the arborescent knot $K_n$. The tangles $T$ and $B$ are exactly as shown. The twist regions below $T$ and above $B$ have $2x_n$ and $2 y_n$ crossings, respectively. The braid $\tau^{x_n} \circ \varphi^{2n} \circ \tau^{y_n}$ that connects $B$ to $T$ is required to be alternating. 
We obtain $L_n$ by replacing the twist regions  $\tau^{x_n}$ and $\tau^{y_n}$ by unknotted circles, $Q_T$ and $Q_B$, as shown. Then $K_n$ can be recovered by  Dehn filling along those circles.}
	\label{Fig:ArbKnots}
\end{figure}
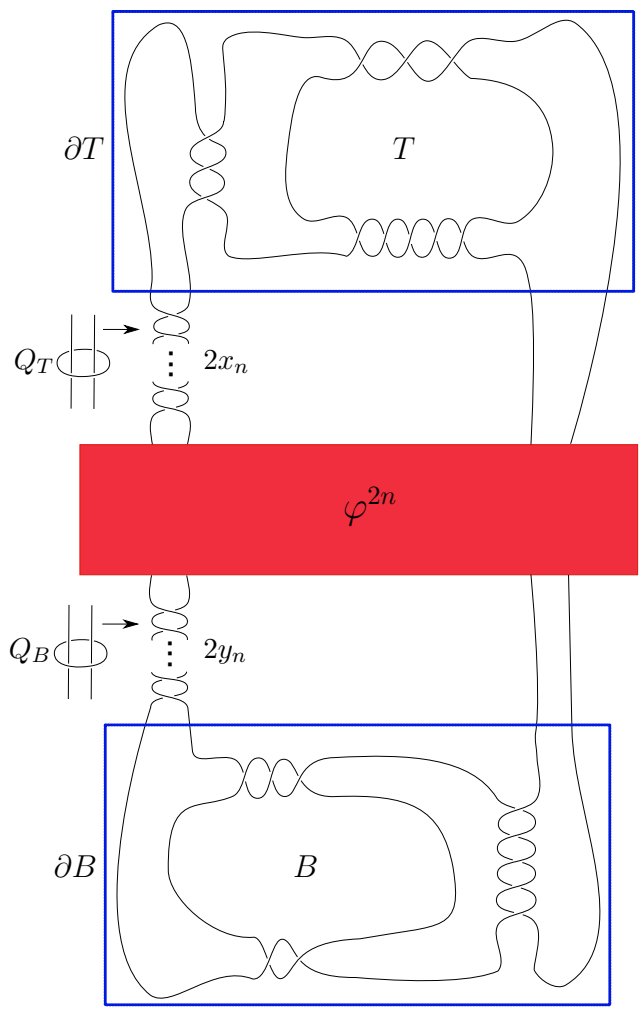

Because of the high degree of similarity to \refsec{GeoSimManifolds}, most of this section will be a sketch. We will focus more attention on the few places where the argument differs.

%%%%%%%%%%%%%%%%%%%%%%%%%%%%%%%%%%%%%%%%%%%%%%%%%%%%%%%%%%%%%%%%%%%%

\subsection{Our construction}
\label{subsec:construction}
For the rest of the section, fix a surface $S = S_{0,4}$. We begin by describing the analogue of caps for $S$.

\begin{definition}\label{Def:Tangle}
A \emph{tangle} is a pair $(D, L)$, where $D$ is the $3$--ball, and $L \subset D$ is a properly embedded $1$--manifold such that $\bdy L \subset \bdy D$ consists of exactly $4$ points. We identify $\bdy D \setminus \bdy L$ with $S_{0,4}$. The tangle $(D,L)$ is called \emph{rational} if $L$ consists of two arcs that are simultaneously boundary--parallel in $D$.
%    
%    We thicken $L$ to a regular neighborhood $\eta(L)$, and endow $D \setminus \eta(L)$ with a pared manifold structure. Here, the paring locus consists of $\bdy \eta(L) \setminus \bdy D$.
\end{definition}

Conway defined the operations of addition and multiplication on tangles, both of which involve joining a pair of tangles along half of their boundaries. A tangle is called \emph{arborescent} if it can be obtained from rational tangles via these operations. See e.g.\ Wu \cite{Wu2} for a clear summary of these operations. The relevant fact in our setting is that the tangles $B$ and $T$ depicted in \reffig{ArbKnots} are arborescent but not rational.

We now perform the following sequence of constructions.

\begin{enumerate}[1.]
\item \underline{Construct $T$ and $B$:} Let $T = (D_T, L_T)$ and $B = (D_B, L_B)$ be exactly the tangles depicted in \reffig{ArbKnots}.
In particular, each of $T$ and $B$ is an arborescent tangle built out of $3$ rational tangles.
Note that in \reffig{ArbKnots}, the $4$ points of $\bdy L_T$ are directly above the $4$ points of $\bdy L_B$. We obtain a model copy of $S = S_{0,4} \subset S^3 \setminus (B \cup T)$ by taking a horizontal plane in $\RR^3$ between $B$ and $T$, puncturing it at the $4$ points below $\bdy L_T$, and compactifying via the point at $\infty$.

\smallskip

\item \underline{Construct $Q$, $\tau$, and $\varphi$:}  Let $Q \subset S$ be the simple closed curve encircling the two left-most punctures in the model copy of $S_{0,4}$. Translating this model surface up and down in $S^3$ gives an unknotted circle $Q_T \subset \bdy T$ and an unknotted circle $Q_B \subset \bdy B$, as shown.
Let $\tau$ be a (left-handed) Dehn twist along $Q \subset S$, or equivalently a left-handed  full twist about the left-most pair of strands passing through the model copy of $S$. 

Let $\varphi$ be a pure braid on $4$ strands in $S^2$, chosen to satisfy the following properties.
First, choose $\varphi$ so that $\tau \circ \varphi \circ \tau$ is an alternating $4$--braid. In addition, choose $\varphi \in$ so that $4 E t_{\Curve}(\varphi) > 1$. Here, $E$ is the constant of \refeqn{LowerElectricDist}, while $t_{\Curve}(\varphi)$ is the translation distance of $\varphi$ in the curve complex $\Curve(S_{0,4})$, as in \refeqn{StableLength}. As discussed in Remark \ref{Rem:CoarseConstants}, this criterion can aways be achieved by taking an (alternating, pure) pseudo-Anosov $\varphi_0$ and letting $\varphi$ be some power of $\varphi_0$.

\smallskip

\item \underline{Construct $K_n$ and $E_n = S^3 \setminus K_n$:} For every $n \geq 1$, we build a knot $K_n \subset S^3$ as follows.
Choose positive integers $x_n$ and $y_n$, so that $x_n \gg y_n \gg 0$. Then, join the $1$--manifold $L_B \subset D_B$ to $L_T \subset D_T$ via the alternating pure braid $\tau^{x_n} \circ \varphi^{2n} \circ \tau^{y_n}$. (If the braid word is read right to left, as a function, the crossings should be read bottom to top.) 

The resulting link $K_n$, depicted in \reffig{ArbKnots}, will be a knot because the pure braid $\tau^{x_n} \circ \varphi^{2n} \circ \tau^{y_n}$ connects the punctures of $S$ in exactly the same way as the empty braid. By construction, $K_n$ is alternating and arborescent. In addition, $E_n = S^3 \setminus K_n$ satisfies the recipe \refeqn{GluingFormula}.

\smallskip

\item \underline{Construct $L_n$, $J_n$ and $W_{\varphi,n}$:} Consider the $3$--component link $L'_n = K_n \cup Q_B \cup Q_T$, and let $J_n = S^3 \setminus L'_n$. Equivalently, $J_n \cong S^3 \setminus L_n$, where $L_n$ is the link obtained from $L'_n$ by removing the full twists $\tau^{x_n}$ and $\tau^{y_n}$.

By construction, $L_n$ is an augmented alternating link, hence $J_n = S^3 \setminus L_n$ is hyperbolic by a theorem of Adams \cite{adams:auglink}. In the hyperbolic metric on $J_n$, the surface $\bdy T \setminus Q_T$ is the union of two totally geodesic pairs of pants, and similarly for $\bdy B \setminus Q_B$. These totally geodesic pairs of pants decompose $J_n$ into caps corresponding to pared tangles $(T, L_T \cup Q_T)$ and $(B, L_B \cup Q_B)$, along with a product region $W_{\varphi,n}$ between the caps. By construction, the paring locus of $\bdy_+ W_{\varphi,n}$ occurs along the punctures of $S$ and $Q_T = \varphi^n Q$, and similarly for $\bdy_- W_{\varphi,n}$.

We may recover $K_n$ from $J_n = S^3 \setminus L_n$, by performing $1/x_n$ surgery along $Q_T$ and $1/y_n$ surgery along $Q_B$.

\smallskip

\item \underline{Construct $K_n^\mu$ and $E_n^\mu = S^3 \setminus K_n^\mu$:} Let $F_n \subset E_n$ be a model copy of $S$, placed between the caps $T$ and $B$. Let $\mu : S \to S$ be the hyper-elliptic involution corresponding to $\pi$--rotation about a horizontal line in $S$. Let $K_n^\mu$ be the knot obtained from $K_n$ by mutation along $F_n$. In \reffig{ArbKnots},  $K_n^\mu$ can be obtained by removing the tangle $B$, rotating it by $\pi$ about a horizontal axis (thus interchanging the two horizontal twist regions), and gluing it back in.

\smallskip 

\item \underline{Construct $X_n$:} Let $X_n$ be the cover of $E_n$ corresponding to $\pi_1(F_n)$. As in \refsec{Mutation}, the hyper-elliptic involution $\mu$ extends over $X_n$, hence the cover of $E_n^\mu$  corresponding to $\pi_1(F_n)$ is isometric to $X_n$.
\end{enumerate}

\subsection{Proving \refthm{GeoSimKnots}}
\label{Sec:KnotProof}

It remains to check that $E_{n} = S^{3} \setminus K_{n}$ and $E_{n}^{\mu} = S^{3} \setminus K_{n}^{\mu}$ have all the properties claimed in \refthm{GeoSimKnots}. Most of the steps follow the same outline as \refsec{GeoSimManifolds}. The one main difference is in \reflem{KnotGeodesics}, where we must pay attention to geodesics that cut through the cusp of $E_n$.

\begin{lemma}\label{Lem:HypArborKnots}
$K_{n}$ and $K_{n}^{\mu}$ are hyperbolic, arborescent, alternating knots. Neither knot complement is arithmetic. Furthermore, neither knot has any lens space surgeries.
\end{lemma}

\begin{proof}%[Proof sketch]
$K_n$ and $K_n^\mu$ are arborescent by construction, because we have connected two arborescent tangles by a $4$--string braid. The knots are alternating by construction, because we have connected two alternating tangles via a braid that alternates in a consistent direction. By the work of Bonahon and Siebenmann \cite{BoSi}, any non-hyperbolic arborescent knot is Montesinos (meaning, it is a cyclic sum of rational tangles, which $K_n$ and $K_n^\mu$ are not). See also Futer--Gu\'eritaud \cite{FG:arborescent} and Wu \cite{Wu2}. Thus $E_n$ and $E_n^\mu$ are hyperbolic.

A theorem of Wu \cite{Wu2} also implies that $K_n$ and $K_n^\mu$ have no exceptional surgeries. 
By a theorem of Reid \cite{Reid:isospectrality}, the only knot in $S^3$ with an arithmetic complement is the figure--$8$ knot (which is Montesinos). Thus $E_n$ and $E_n^\mu$ are non-arithmetic.
\end{proof}

\begin{lemma}\label{Lem:KnotVolume}
$\vol(E_{n}) = \vol(E_{n}^{\mu})$, and this volume grows coarsely linearly with $n$. \end{lemma}

\begin{proof}%[Proof sketch]
Since $K_n$ and $K_n^\mu$ are mutant knots, their complements $E_n$ and $E_n^\mu$ have equal volume by Ruberman's theorem \cite{Ru}. 

By the same Dehn filling argument as in \refprop{NnGeometry}, $\vol(E_n) = \vol(W_{\varphi,n}) + O(1)$, and  $\vol(W_{\varphi,n})$ grows linearly with $n$ by \refprop{LinearGeometryGrowth}. We remark that in the present context, where $S = S_{0,4}$, one could obtain explicit upper and lower bounds on $\vol(E_n)$ from the work of Gu\'eritaud and Futer \cite[Appendix B]{GuFu}.
\end{proof}

\begin{lemma}\label{Lem:AsymmetricKnots}
The pared tangles $(T, L_T \cup Q_T)$ and $(B, L_B \cup Q_B)$ are simple, distinct, and asymmetric. It follows that for $n \gg 0$, the knot complements  $E_n = S^3 \setminus K_n$ and $E_n^\mu = S^3 \setminus K_n^\mu$ are distinct and asymmetric.
\end{lemma}

\begin{proof}
The conclusion about the tangles is checked using SnapPy, as in \reflem{AsymmetricCaps}. Note that doubling $T \setminus Q_T$ along its geodesic boundary produces a link $(L_T \cup \overline{L_T} \cup Q_T) \subset S^3$, where $\overline{L_T}$ is the mirror image of $L_T$. This link can be analyzed using SnapPy, verifying that it has no symmetries apart from reflection in the totally geodesic surface $\bdy T \setminus Q_T$. Regina checks that the only essential pants in $(T, L_T \cup Q_T)$ occur along the boundary. The same process works for the bottom tangle $B$.

Now, an argument using pairs of pants (just as in \reflem{MutantVolume}) implies that $J_n$ and $J_n^\mu$ are non-isometric. Similarly, any self-isometry of $J_n$ would have to preserve its decomposition into $B$, $T$, and $W_{\varphi,n}$, implying that the map is the identity. Because we have chosen very long Dehn filling slopes along $Q_T$ and $Q_B$, any homeomorphism $E_n \to E_n^\mu$ or $E_n \to E_n$ would restrict to a homeomorphism (hence an isometry) on $J_n$, a contradiction.
\end{proof}

We remark that the distinctness and asymmetry of $K_n$ and $K_n^\mu$ can also be checked using knot-theoretic tools. See the classification of arborescent knots by Bonahon and Siebenmann \cite[Theorems 12.12 and 16.4]{BoSi}, or the classification of alternating knots by Menasco and Thistlethwaite \cite{menasco-thistlethwaite}.
%    
%    This approach requires constructing so-called \emph{canonical weighted trees} for $K_n$ and $K_n^\mu$, but 
This approach does not need the hypothesis that $n \gg 0$.

\begin{lemma}\label{Lem:KnotUnique}
For $n \gg 0$, and an appropriate choice of $x_n \gg y_n \gg 0$, the following holds. Each of $E_{n}$ and $E_{n}^{\mu}$ is the unique minimal orbifold in its commensurability class, and the only knot complement in its commensurability class.
\end{lemma}

\begin{figure}
	\begin{overpic}[width=4in]{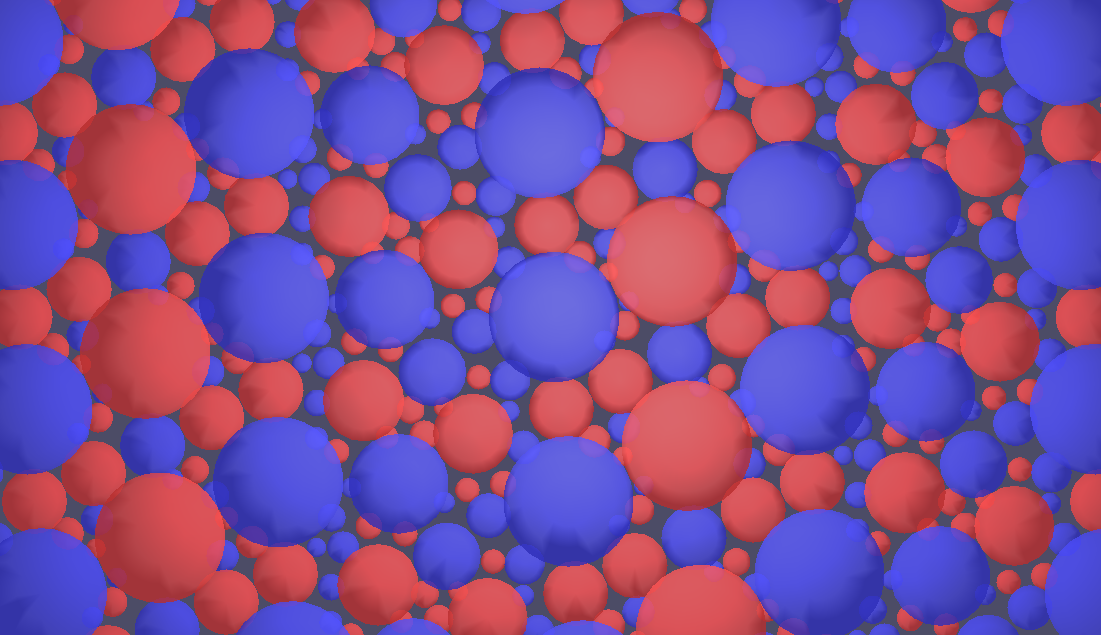}
	\end{overpic}
	\caption{A horoball diagram for the pared version of the top tangle $(T, L_T \cup Q_T)$. The view is from the cusp of $K_n$. The blue horoballs correspond to arcs from $K_n$ to $Q_T$. The totally geodesic pairs of pants $P$ and $P'$ lift to nearly vertical lines in this figure. The pattern of horoballs cannot be invariant under a rotation of order $3$ or $4$, implying that $M_n = E_n \setminus Q_T $ cannot cover an orbifold with a rigid cusp.	}
	\label{Fig:TopTangleHoroballs}
\end{figure}

\begin{proof}
By \reflem{HypArborKnots}, $E_n$ is non-arithmetic, hence \refthm{MargulisCommensurator} says that it covers a minimal orbifold $\calQ_n = \HH^3 / C^+(E_n)$. Suppose that the cover $\psi: E_n \to \calQ_n$ is non-trivial. Then $\psi$ must be an irregular cover, because $E_n$ has no symmetries by   \reflem{AsymmetricKnots}. In other words, $\psi$ is given by a so-called \emph{hidden symmetry}. By a theorem of Neumann and Reid \cite[Proposition 9.1]{NeRe}, the cusp neighborhood of $\calQ_n$ is \emph{rigid}, meaning that its cross-section is a rigid $2$--orbifold (recall \refdef{EucOrbifold}).

As in the proof of \refthm{PantedCriterion}, we choose the filling coefficients $x_n \gg y_n \gg 0$ sufficiently long and sufficiently distinct to ensure that each of the core curves created by Dehn filling $Q_T$ and $Q_B$ is the complete preimage of its image in $\calQ_n$. Thus the irregular covering map $\psi: E_n \to \calQ_n$ restricts to a covering map $\psi: M_n \to \calO_n$, where $M_n = E_n \setminus Q_T$. Each of the two cusps of $M_n$ covers a distinct cusp of $\calO_n$. Furthermore, by the above paragraph, the cusp of $M_n$ corresponding to $K_n$ must cover a rigid cusp of $\calO_n$.

%    Now, we argue as in the proof of \reflem{NotRigid}. 
Observe that $M_n$ contains an isometric copy of the pared tangle $(T, L_T \cup Q_T)$, with boundary along two pairs of pants, $P$ and $P'$. 
We cannot apply \refthm{PantedCriterion} directly, because $P$ and $P'$ are not pairwise tangent (see \reffig{TopTangleHoroballs}). Nevertheless, that figure shows that the horoball diagram for the top tangle  $(T, L_T \cup Q_T)$  cannot be invariant under a rotation by angle $2\pi/3$ or $\pi/2$. On the other hand, every rigid $2$--orbifold contains a cone point of order $3$ or $4$ (see \refdef{EucOrbifold} and \reflem{NotRigid}). Thus the quotient orbifold $\calO_n$ cannot have a rigid cusp, which is a contradiction.

This contradiction implies that  the cover  $E_n \to \calQ_n$ is trivial, hence $E_n$ has no symmetries or hidden symmetries, and is minimal in its commensurability class. In addition, by \reflem{HypArborKnots}, $E_n$ has no lens space surgeries. Now, a theorem of Reid and Walsh  \cite[Proposition 5.1]{ReWa} implies that $E_n$ is the only knot complement in its commensurability class. The same argument applies to $E_n^\mu$.
\end{proof}

\begin{lemma}\label{Lem:KnotGeodesics}
For $n \gg 0$, any closed geodesic $\gamma \subset E_n$ whose length is at most $2 \log n$ can be homotoped to be disjoint from the mutation surface $F_n$. Consequently, there is a bijection between the complex length spectra of $E_n$ and $E_n^\mu$ up to length $2 \log n$.
\end{lemma}

\begin{proof}
This follows by the same argument as in \refprop{NnGeometry} and \reflem{SameGeodesicsMutation}. However, we need to pay more attention to the cusp of $E_n$.

By \refprop{LinearGeometryGrowth}, and our choice of $\varphi$ as in \refrem{CoarseConstants}, the manifold $W_{\varphi,n}$ satisfies
\[
d_0 (\bdy_- \core^0 W_{\varphi,n}, \, \bdy_+ \core^0 W_{\varphi,n})  > n / 2.
\]
Recall that for a general hyperbolic manifold $M$, and the Margulis constant $\epsilon$, $\core^0 (M)$ is the complement of the $\epsilon$--thin neighborhood of the cusps of $M$.
Furthermore, $d_0(\cdot, \cdot)$ is the shortest length of a path from the lower boundary of $\core^0 W_{\varphi,n}$ to the upper boundary of $\core^0 W_{\varphi,n}$, among paths that remain inside $\core^0 W_{\varphi,n}$.

As in \refprop{NnGeometry}, we may construct pleated surfaces $\Sigma_\pm \subset E_n$ whose geometry closely approximates $\bdy_\pm \core \, W_{\varphi,n}$. By the argument  in that proposition (using the Brock--Bromberg bilipschitz theorem \cite{brock-bromberg:density}), the portion of these surfaces in $\core^0 E_n$ satisfies
\begin{equation}\label{Eqn:EnLength}
d_0 (\Sigma_-, \, \Sigma_+)  > n/2.
\end{equation}
As above, $d_0(\cdot,\cdot)$ only considers paths that remain in $\core^0 E_n$.

Let $\gamma \subset E_n$ be a closed geodesic of length at most $2 \log n$. We claim that $\gamma$ can be homotoped to be disjoint from at least one of $\Sigma_-$ and $\Sigma_+$. Note that Equation~\refeqn{EnLength} is not immediately applicable, because $\gamma$ may fail to be entirely contained in $\core^0 E_n$. However, $\gamma$ must have non-trivial intersection with $\core^0  E_n$, because the $\epsilon$--thin horocusp $C_n = E_n \setminus \core^0 E_n$ contains no closed geodesics.
%Let $C_n \subset E_n$ be the $\epsilon$--thin neighborhood of the cusp.

If $\gamma \subset \core^0 E_n$, we already have $\ell(\gamma)/ 2 \leq \log n < n/2$. 
Otherwise, if $\gamma$ is not contained in $\core^0 E_n$, we may break it up into geodesic sub-arcs $\gamma_1, \ldots, \gamma_{2k}$,  such that every odd-numbered $\gamma_i$ is contained in $\core^0 E_n$ and every even-numbered $\gamma_i$ is contained in  $C_n$.
% (In case $\gamma \subset \core^0 E_n$, we set $k = 1$ and let $\gamma_2 = \gamma'_2$ be an arc of length $0$.) 
Every even-numbered geodesic arc $\gamma_i$ is homotopic to a horocyclic segment $\gamma'_i \subset \bdy C_n$ with the same endpoints. By \cite[Lemma A.2]{cfp:tunnels}, 
\[
\frac{ \ell(\gamma'_i) }{ 2}  = \sinh \frac{ \ell(\gamma_i) }{ 2 }.
\]
Now, construct a closed loop $\gamma' = \gamma_1 \cdot \gamma'_2 \cdot \ldots \cdot \gamma_{2k-1} \cdot \gamma'_{2k}$, where $\cdot$ denotes concatenation. By construction, $\gamma'$ is homotopic to $\gamma$, contained in $\core^0 E_n$, and its length satisfies
\begin{align*}
\frac{ \ell(\gamma') }{2}
& = \sum_{\text{odd } i} \frac{ \ell(\gamma_i) }{2} +  \sum_{\text{even } i} \frac{ \ell(\gamma'_i) }{2} \\
& < \sum_{\text{odd } i} \sinh \frac{ \ell(\gamma_i) }{2} +  \sum_{\text{even } i} \sinh \frac{ \ell(\gamma_i) }{2} \\
& <  \sinh \left(  \sum_{ i=1}^{2k}  \frac{ \ell(\gamma_i) }{2} \right) \\
& =  \sinh \frac{ \ell(\gamma) }{2} \\
& < \frac{1}{2} \exp \big( \ell(\gamma) / 2 \big) \\
& \leq n/2.
\end{align*}

Since $d_0 (\Sigma_-, \, \Sigma_+)  > n/2$, and the surfaces $\Sigma_\pm$ are (homologically) separating, the closed loop $\gamma'$ must be disjoint from either $\Sigma_-$ or $\Sigma_+$. As a consequence, the geodesic $\gamma$ (which is homotopic to $\gamma'$) can be homotoped to be disjoint from $F_n$. 

Since all geodesics of length at most $2 \log n$ can be homotoped to be disjoint from $F_n$, the conclusion about length spectra follows by a theorem of Millichap \cite[Proposition 4.4]{Mi2}.
\end{proof}

\begin{lemma}\label{Lem:KnotGeodesicCount}
For $n \gg 0$, each of $E_n$ and $E_n^\mu$ contains at least $n^2/(2 \log n)$ closed geodesics up to length $2 \log n$. 
\end{lemma}

\begin{proof}
This is proved by exactly the same argument as in \reflem{GeodesicCountMutation}. Applying \refprop{XnGeodesics} with the length cutoff $L = 2\log n$ gives the desired conclusion for the manifold $X_n$ that covers both $E_n$ and $E_n^\mu$. Since the top and bottom tangles $T$ and $B$ are simple by \reflem{AsymmetricKnots}, the same conclusion holds for  $E_n$ and $E_n^\mu$.
\end{proof}

Lemmas~\ref{Lem:KnotVolume}--\ref{Lem:KnotGeodesicCount} complete the proof of \refthm{GeoSimKnots}. \qed

%%%%%%%%%%%%%%%%%%%%%%%%%%%%%%%%%%%%%%%%%%%%%%%%%%%%%%%%%%%%%%%%%%%%%%%%

\bibliographystyle{hamsplain}
\bibliography{biblio}

\end{document}